\documentclass[twoside]{amsart}

\usepackage{amsmath,amssymb,amsfonts,amsthm,latexsym}
\usepackage{amscd,graphicx,color,enumerate}
\usepackage{hyperref}
\usepackage[margin=1in]{geometry}
\usepackage{mathrsfs}
\usepackage{tikz,url}
\usepackage[all]{xy}
\usepackage{enumitem}
\usepackage{stackengine}
\stackMath
\usepackage{orcidlink}

\numberwithin{equation}{section}

\newtheorem{theorem}{Theorem}[section]
\newtheorem{lemma}[theorem]{Lemma}
\newtheorem{proposition}[theorem]{Proposition}
\newtheorem{corollary}[theorem]{Corollary}

\theoremstyle{definition}

\newtheorem{example}[theorem]{Example}
\theoremstyle{remark}
\newtheorem{remark}[theorem]{Remark}

\newcommand{\C}{\mathbb{C}}

\newcommand{\Q}{\mathbb{Q}}

\newcommand{\Z}{\mathbb{Z}}

\newcommand{\Rank}{\operatorname{rank}}

\begin{document}

\title{Pivotal condensation and chemical balancing}

\author{Hans-Christian Herbig\,\orcidlink{0000-0003-2676-3340}}
\email{herbighc@gmail.com}
\address{Departamento de Matem\'{a}tica Aplicada, Universidade Federal do Rio de Janeiro,
Av. Athos da Silveira Ramos 149, Centro de Tecnologia - Bloco C, CEP: 21941-909 - Rio de Janeiro, Brazil}


\begin{abstract} 
I present a universal method, called pivotal condensation, for calculating stoichiometric factors of chemical reactions. It can be  done by hand, even for rather complicated reactions. The main trick, which I call \emph{kernel pivotal condensation} (ker pc), to calculate the kernel of a matrix might be of independent interest. The discussion is elaborated for matrices with entries in a principal ideal domain $R$. The ker pc calculates a basis with coefficients in $R$ for the kernel of a matrix, seen as a $Q$-vector space, where $Q$ is the quotient field of $R$.
If $W$ is a free saturated $R$-submodule of $R^n$ I address the question how to modify an $R$-basis of the $Q$-vector subspace $Q\otimes _R W$ over the quotient field $Q$ to obtain a basis of the $R$-module $W$.
I also indicate how one can solve inhomogeneous linear systems, invert matrices and determine the four subspaces using 
pivotal condensation. I formulate the balancing by inspection method that is widely used to reduce the size of a linear system arising in chemical balancing in mathematical language.
 \end{abstract}
\maketitle

\tableofcontents


\section{Introduction}
\label{sec:Intro}

While teaching the basic linear algebra course to chemical engineering students I noticed that for chemical balancing (see, e.g., \cite{Fink,whynots} and references therein) the naive methods, in particular balancing by inspection (see \cite[Subsection 13.2]{Fink}), acquired in the chemistry classes are for all practical purposes way more efficient than applying Gaussian elimination to the corresponding homogeneous linear system. 
For complicated reactions the naive methods can become cumbersome and require skill.
The question arose if there is a universal method that is relatively quick and safe for a human computer. I am presenting such a method that works over the integers. It is inspired by the pivotal condensation of Felice Chiò \cite{Chio} for calculating determinants of square matrices.

In chemical balancing it sometimes makes sense to balance chemical reactions that depend on a parameter, say $x$. This is for example the case if one has families of compounds. In this situation one has to do linear algebra over the gcd domain $R=\Z[x]$, which is not even a principal ideal domain (PID). The suggested algorithms for the kernel of a matrix are elaborated for the case when $R$ is a gcd domain (e.g. a PID). Readers unfamiliar with abstract algebra can assume  that the base ring is $R=\Z$ with quotient field $Q=\Q$. I emphasize that the pivotal condensation provides a basis for the kernel of a matrix with entries in $R$
that has coefficients in $R$ seen as a vector space over the quotient field $Q$. 
But if  the dimension of the kernel is $>1$ it may happen that the basis is not a basis of the kernel seen as a free $R$-module. This subtlety is typically swept under the rug in the literature on chemical balancing and for the convenience of the reader I spell out in the more mathematical Section \ref{sec:Rbases} how this can be handled in practical terms when $R$ is a PID.
 The only other work that I am aware of that views the chemical balancing problem as a system of linear Diophantine equations is \cite{Papp}. Their empirical analysis relies on computer implementations of certain algorithms. However, I am trying here to convince the reader that all practically relevant examples can be elaborated by a patient human on a sheet of paper.

The plan of the paper is as follows. In Section \ref{sec:PC} I review the ideas of pivotal condensation. 
In Subsection \ref{subsec:Chio} I recall the pivotal condensation method to calculate the determinant $ |A|$ of a square matrix $A\in R^{n\times n}$. In Subsection \ref{subsec:ker} I propose a variation of pivotal condensation that allows to re-purpose parts the findings of \ref{subsec:Chio} to determine a basis for the kernel of the transposed matrix $A^{T}$ in case $A$ turns out to be singular.
For non-square matrices $A^T$ I simplify the scheme notationally and calculate merely $\ker\left( A\right)$. In Subsection \ref{subsec:hunt} I propose consistency checks that can be utilized to localize computational errors. With this preparation we are ready to turn in Section \ref{sec:CB} to chemical balancing and work out some examples. In Section \ref{sec:quivered} I formulate the balancing by inspection method in mathematical language. In Section \ref{sec:misc} I show how to solve inhomogeneous systems, invert matrices and determine the four subspaces 
using pivotal condensation.
In the Section \ref{sec:proofs} I give a 
formal justification of the pivotal condensation scheme. 
In Section \ref{sec:Rbases} I review how $R$-module bases can be calculated by hand. This includes a discussion of the saturation of an $R$-module, a practical guide to Smith normal form and a quick method to post-process the result of the pivotal condensation
the obtain an $R$-module basis for the kernel.

The practitioners among the readership may focus on Sections  \ref{subsec:hunt}, \ref{sec:CB} and, possibly, \ref{sec:misc} and ignore the rest of the paper. To do kernel pivotal condensation (ker pc) and, for example, matrix inversion with ker pc, by hand requires only the knowledge of arithmetic acquired in 5th grade and some patience. One does not need to use the calculus of fractions.  To have a completely satisfactory answer the results of ker pc have sometimes to be post-processed by a Smith normal form calculation (see Subsection \ref{subsec:Rbasis}). This requires division with remainder. I give a detailed account of the example calculations for those readers with a moderate mathematical background who just want to learn the method.

\hspace{2mm}

\noindent\textbf{Acknowledgements.} I would like thank the undergraduate crowd at the Escola de Química da Universidade Federal do Rio de Janeiro for helping me to think about the matter, mostly by not paying attention to what I said. The paper is an attempt towards a more just and equitable chemistry. Let us bring the wonders of stoichiometry even to the remotest places of the Amazon rain forest without polluting them with consumer electronics. Those who have seen the movie \emph{The Great Race} know that `Push the button, Max!' is an elitist and dangerous attitude. I would like to thank Benjamin Briggs for pointing out to me the notion of saturatedness. The whole project was triggered by a chat \cite{Grossman} that I had with Ben Grossman on \texttt{stackexchange.com}.
Thanks to Daniel Herden and Chris Seaton for saving me from embarrassments and Ihsen Yengui for a clarification.

\hspace{2mm}

\noindent\textbf{Statements and Declarations.} The author declares that he has no conflict of interest. He is a member of the \emph{Centro Brasileiro de Geometria} of the \emph{Programa Institutos Nacionais de Ciência e Tecnologia (INCT)}. 

\section{Pivotal condensation}
\label{sec:PC}
Chiò's \emph{pivotal condensation (pc)} is based on the following determinantal identity (for some historical account of the subject see \cite{Abeles}). 
\begin{theorem} Let $n\geq 2$ be an integer, $R$ an integral domain with quotient field $Q$ and $A=[a_{ij}]_{i,j=1,\dots,n}\in R^{n\times n}$ with $a_{11}\neq 0$. Then
\begin{align}\label{eq:Chio}
|A|=\frac{1}{a_{11}^{n-2}}\begin{vmatrix}
\begin{vmatrix}
a_{11} & a_{12}\\
a_{21} & a_{22}
\end{vmatrix} & \begin{vmatrix}
a_{11} & a_{13}\\
a_{21} & a_{23}
\end{vmatrix} &\cdots  & \begin{vmatrix}
a_{11} & a_{1n}\\
a_{21} & a_{2n}
\end{vmatrix}\\
\addstackgap{
\begin{vmatrix}
a_{11} & a_{12}\\
a_{31} & a_{32}
\end{vmatrix}} & \begin{vmatrix}
a_{11} & a_{13}\\
a_{31} & a_{33}
\end{vmatrix} &  & \begin{vmatrix}
a_{11} & a_{1n}\\
a_{31} & a_{3n}
\end{vmatrix}\\
\cdots &  &  & \cdots\\
\begin{vmatrix}
a_{11} & a_{12}\\
a_{n2} & a_{n2}
\end{vmatrix} & \begin{vmatrix}
a_{11} & a_{13}\\
a_{n2} & a_{n3}
\end{vmatrix} & \cdots & \begin{vmatrix}
a_{11} & a_{1n}\\
a_{n2} & a_{nn}
\end{vmatrix}
\end{vmatrix}.
\end{align}
\end{theorem}
Formally this expression makes only sense in $Q$, but of course $|A|\in R$. A straight forward proof of Equation \eqref{eq:Chio} in terms of row operations can be found in \cite[Paragraph 3.6.1]{Eves}. I am reproducing here an argument from \cite{Abeles} generalizing Equation \eqref{eq:Chio}.
The generalization is based on a block matrix decomposition of Sylvester:
\begin{align*}
    A=\begin{bmatrix}
    A_{11} & A_{12}\\
A_{21} & A_{22}
\end{bmatrix} =\begin{bmatrix}
A_{11} & \boldsymbol{0}\\
A_{21} & \boldsymbol{1}
\end{bmatrix}\begin{bmatrix}
\boldsymbol{1} & A_{11}^{-1} A_{12}\\
\boldsymbol{0} & A_{22} -A_{21} A_{11}^{-1} A_{12}
\end{bmatrix}.
\end{align*}
Here $n=k+l$, $A_{11}\in R^{k\times k}$ is assumed to be invertible, $A_{12}\in R^{k\times l}$, $A_{21}\in R^{l\times k}$ and $A_{22}\in R^{l\times l}$.
Taking determinants we deduce $|A|=\left|A_{11}\right|\left|  A_{22} -A_{21} A_{11}^{-1} A_{12}\right|$. For $k=1$ this means $|A|=A_{11}^{-n+2}\left|  A_{11}A_{22} -A_{21} A_{12}\right|$ since $A_{11}$ is central, proving \eqref{eq:Chio}.
\section{Determinant}
\label{subsec:Chio}

The pc for evaluating determinants, referred to as \emph{Chiò's pc}, consists of repeated applications of the following loop in order to reduce the size of the determinant.

\noindent\textbf{Step 1 (preparation).} If $a_{11} =0$ search for an $i$ with $a_{i1} \neq 0$. Switch the $i$th row and the $1$st row  and keep track of the sign. If $a_{i1} =0$ for all $i=1,\dotsc ,n$ then stop; in fact, $|A|=0$.

\noindent\textbf{Step 2 (pc).} Apply formula \eqref{eq:Chio} to write $|A|$ as a scalar multiple of a $( n-1) \times ( n-1)$-determinant $|A'|$.

\noindent\textbf{Step 3 (cleaning up).} Pull out of the determinant $|A'|$ common factors of rows or columns (step 3 is not compulsory). 

Continue the process recursively with the $(n-1) \times (n-1)$-determinant etc., i.e., apply repeatedly Steps 1,2,3, until obtaining a $2\times 2$-determinant, which
 in turn is to be evaluated. For clarity I put in Step 1 a box around the pivot $a_{i1}$. 

\begin{example}
To give an example, let us calculate the area $\mathcal A$ of the triangle with the sides $2,3,4$ using the Cayley-Menger determinant
(see, e.g., \cite[Appendix A.1]{Fiedler} )
\begin{align*}
\begin{vmatrix}
0 & 1 & 1 & 1\\
1 & 0 & 2^2 & 3^{3}\\
1 & 2^{2} & 0 & 4^{2}\\
1 & 3^{2} & 4^{2} & 0
\end{vmatrix}\overset{\text{Step 1}}{=} -\begin{vmatrix}
\boxed{1} & 9 & 16 & 0\\
1 & 0 & 4 & 9\\
1 & 4 & 0 & 16\\
0 & 1 & 1 & 1
\end{vmatrix}\overset{\text{Step 2}}{=} -\frac{1}{1^{4-2}}\begin{vmatrix}
-9 & -12 & 9\\
-5 & -16 & 16\\
1 & 1 & 1
\end{vmatrix}\overset{\text{Step 3}}{=} -3\begin{vmatrix}
\boxed{-3} & -4 & 3\\
-5 & -16 & 16\\
1 & 1 & 1
\end{vmatrix}\\
\overset{\text{Step 2}}{=} -\frac{3}{( -3)^{3-2}}\begin{vmatrix}
28 & -33\\
1 & -6
\end{vmatrix}\overset{\text{Step 3}}{=} 3\begin{vmatrix}
28 & -11\\
1 & -2
\end{vmatrix} =3( -45) =-135=-( 2!)^2 2^2{\mathcal A}^{2}.
\end{align*}
Here the ring $R$ can be taken to be the integers $\Z$.
We conclude that $ \mathcal A=\sqrt{135}/4$.
\end{example}
The scheme provides a compact and safe means to evaluate determinants by hand. Typically the symmetry of the matrix can be preserved in the process, cutting the workload almost in half. It combines excellently with Cramer's rule as most of the minors that appear in the denominator show up in the numerator as well.

\section{Kernel}
\label{sec:kernel}
When one is confronted with a quadratic matrix $A\in R^{n\times n}$ it is oftentimes unclear if it is regular or singular. In the latter case one is possibly expected to work out the kernel $\ker_Q(A) $ as well. Throughout the paper we use the notation
\begin{align*}
\ker_R(A)&=\{\vec v\in R^n\ | \ A\vec v=\vec 0\},  \ \ \ \operatorname{im}_R(A)=\{A\vec v\in R^n\ | \vec v\in R^m\}=AR^n,\\
    \ker_Q(A)&=\{\vec v\in Q^n\ | \ A\vec v=\vec 0\}\simeq Q\otimes_R \ker_R(A),  \ \ \ \operatorname{im}_Q(A)=\{A\vec v\in Q^n\ | \vec v\in Q^m\}=AQ^n.
\end{align*}
I present here a version of the pc, referred to as the \emph{detker pc}, that allows to re-purpose findings of Chiò's pc for $|A^T|$ to determine the kernel $\ker_Q(A)$ in case the matrix $A$ turns out to be singular.

At the start of the $(l+1)$st condensation, $l=0,1,\dots,\Rank(A)-2$, one has the pattern 
\begin{align*}
\Delta_l | X_l |Y_l, 
\end{align*}
where $\Delta_l\in Q$, $X_l\in R^{(n-l)\times(n-l)}$ and $Y_l\in R^{(n-l)\times n}$. Here $\Delta_l$ is the product of all the prefactors occuring in the previous steps of the pc, i.e., the signs arising from switching rows, the fractions from Equation \eqref{eq:Chio}, and the common factors pulled out from rows of the matrix. The matrix $X_l$ is the one arising at Chiò's $l$th condensation. The matrix $Y_l$ is capturing the relevant information of the elementary row operations that are involved. At the start of the recursion we put $\Delta_0=1$, $X_0:=A^T$ and $Y_0:=\boldsymbol 1_n$. At the $(l+1)$st condensation the procedure consists of the following.

\noindent\textbf{Step 1 (preparation).} If $X_l=\boldsymbol 0$ then stop; $\Rank(A^T)=l$ and $\ker_Q(A^T)=Y_l^TR^{n-l}$.  Let $k_{l+1}+1$ be the smallest column index of a nonzero column in $X_l$. This means that the first $k_{l+1}$ columns of $X$ are zero. Pick a row index $i$ with $x_{ik_{l+1}}\neq 0$ and replace $\Delta_l\mapsto \Delta_l'=(-1)^{|\sigma|}\Delta_l$, $X_l\mapsto X_l'$, $Y_l\mapsto Y_l'$ where $X_l'$ and $Y_l'$ are obtained from $X_l$ and $Y_l$, respectively, by applying a row permutation $\sigma$ to move the $i$th row to the first row. The resulting pattern is $\Delta_l' | X_l' |Y_l'$. For clarity I put a box around the pivot $x_{ik_{l+1}}$.

\noindent\textbf{Step 2 (pc).} The matrix $Z_l'':=[X_l'' |Y_l'']$  is obtained from $Z_l':=[X_l' |Y_l']$ as follows.
The entry of $Z_{l}''\in R^{( n-l-k_l-1) \times ( 2n-l-1)}$ at the $r$th row and $s$th column is 
\begin{align}
\label{eq:zij}
\begin{vmatrix}
z_{1\, k_{l+1}+1} & z_{1\, k_{l+1}+s}\\
z_{r+1\, k_{l+1}+1} & z_{r+1\, k_{l+1}+s}
\end{vmatrix}, 
\end{align}
where $Z_l'=:[z_{ij}]_{\substack{
i=1,\dotsc ,n-l\\
j=1,\dotsc ,2n-l}}$. The resulting pattern is $\Delta_l''| X_l'' |Y_l''$  with $\Delta_l'':=\Delta_l'z_{1\, k_{l+1}+1}^{-m+l+1}$.

\noindent\textbf{Step 3 (cleaning up).}\footnote{This is optional up to the last iteration.} Let $L_i$ be the $i$th row of $Z_l'':=[X_l'' |Y_l'']$ and $c_i$ be the gcd of the entries of $L_i$ if $L_i$ is nonzero and put $c_i=1$  otherwise. Note that for a general integral domain $R$ the existence of the gcd is not guaranteed. If $R$ is a PID it can be calculated using Euclid's algorithm. Let $Z_{l+1}:=[X_{l+1} |Y_{l+1}]$ be the matrix whose rows are $L_i/c_i$ and put $\Delta_{l+1}:=\Delta_{l}''\prod_i c_i$. 
The resulting pattern is $\Delta_{l+1} | X_{l+1} |Y_{l+1}$, completing the $(l+1)$th condensation.

For each $0\le l\le n$ we have $|A|=\Delta_{l}|X_{l}|$. If $X_{n}\in R$ is nonzero then $\ker_Q(A^T)$ is trivial. In practical terms it seems reasonable to work out the $\Delta_l | X_l |$'s first and leave some blank space to fill in the $Y_l$'s later in case the matrix turns out to be singular.

\begin{example} Let $R=\Z$ and consider the following pc
\begin{align*}
&\begin{array}{ c|c c c c|c c c c }
 & \boxed{1} & 2 & 3 & 4 & 1 & 0 & 0 & 0\\
 & 5 & 6 & 7 & 8 & 0 & 1 & 0 & 0\\
 & 9 & 10 & 11 & 12 & 0 & 0 & 1 & 0\\
1 & 13 & 14 & 15 & 16 & 0 & 0 & 0 & 1
\end{array} \mapsto \begin{array}{ c|c c c|c c c c }
 & \boxed{-4} & -8 & -12 & -5 & 1 & 0 & 0\\
 & -8 & -16 & -24 & -9 & 0 & 1 & 0\\
1 & -12 & -24 & -36 & -13 & 0 & 0 & 1
\end{array}\\
&\mapsto \begin{array}{ c|c c|c c c c }
 & 0 & 0 & -4 & 8 & -4 & 0\\
-1/4 & 0 & 0 & -8 & 12 & 0 & -4
\end{array} \mapsto \begin{array}{ c|c c|c c c c }
 & 0 & 0 & -1 & 2 & -1 & 0\\
-1/4 & 0 & 0 & -2 & 3 & 0 & -1
\end{array}.
\end{align*}
It shows that $\ker_\Q\left(\begin{bmatrix}
1 & 5 & 9 & 13\\
2 & 6 & 10 & 14\\
3 & 7 & 11 & 15\\
4 & 8 & 12 & 16
\end{bmatrix}\right) =\mathbb{Q}\begin{bmatrix}
-1\\
2\\
-1\\
0
\end{bmatrix} \oplus \mathbb{Q}\begin{bmatrix}
-2\\
3\\
0\\
-1
\end{bmatrix}$, which by accident coincides with the kernel of the transposed matrix. As the gcd $\Delta_2$ of the $2\times 2$-minors of 
\begin{align*}
\begin{bmatrix}
   -1&-2\\
2&3\\
-1&0\\
0 &-1
\end{bmatrix}
\end{align*}
is a unit in $\Z$ the above generators also form a basis of the $\Z$-module $\ker_\Z(A)$ (see Theorem \ref{thm:characterization}).
\end{example}
\begin{example} Let $R$ be the ring of convergent power series $\C\{z\}$ in the variable $z$. The following pc
\begin{align*}
&    \begin{array}{ c|c c c|c c c }
 & \boxed{1} & \cos( z) & \cos( 2z) & 1 & 0 & 0\\
 & \cos( z) & 1 & \cos( z) &  0& 1 & 0\\
1 & \cos( 2z) & \cos( z) & 1 & 0 & 0 & 1
\end{array}\\
&\mapsto \begin{array}{ c|c c|c c c c }
 & \sin^{2}( z) & \cos( z)( 1-\cos( 2z)) & -\cos( z) & 1 & 0 & 0\\
1 & \cos( z)( 1-\cos( 2z)) & \sin^{2}( 2z) & -\cos( 2z) & 0 & 1 & 1
\end{array}\\
&=\begin{array}{ c|c c|c c c }
 & \boxed{\sin^{2}( z)} & 2\cos( z)\sin^{2}( z) & -\cos( z) & 1 & 0\\
1 & 2\cos( z)\sin^{2}( z) & 4\sin^{2}( z)\cos^{2}( z) & -\cos( 2z) & 0 & 1
\end{array}\mapsto\\
& \begin{array}{ c|c|c c c }
\sin^{-2}( z) & 4\sin^{4}( z)\cos^{2}( z) -4\sin^{4}( z)\cos^{2}( z) & -\sin^{2}( z)\cos( 2z) +2\cos^{2}( z)\sin^{2}( z) & -2\cos( z)\sin^{2}( z) & \sin^{2}( z)
\end{array}\\
&=\begin{array}{ c|c|c c c }
1 & 0 & 1+\sin^{2}( z) & -2\cos( z) & 1
\end{array}
\end{align*}
shows that $\ker_{\C\{z\}}\left(\begin{bmatrix}
1 & \cos( z) & \cos( 2z)\\
\cos( z) & 1 & \cos( z)\\
\cos( 2z) & \cos( z) & 1
\end{bmatrix}\right) =\C\{z\}\begin{bmatrix}
1+\sin^{2}( z)\\
-2\cos( z)\\
1
\end{bmatrix}\subset \C\{z\}^3$. Here we use the fact that since
$\C\{z\}$ is a PID our vector generates $\ker_{\C\{z\}}$ if and only if it generates $\ker_Q$, where $Q$ is the field of convergent Laurent series (see Remark \ref{rem:1dim}).
\end{example}

If the matrix $A^T\in R^{n\times m}$ is not quadratic I suggest to apply the same procedure to calculate $\ker_Q(A)$ while not annotating $\Delta_l$ and the first $|$. The scheme is referred to as the \emph{ker pc}.  The input of the $(l+1)$th condensation, $l=0,1,\dots,\Rank(A)-2$, one is the pattern 
\begin{align*}
X_l |Y_l, 
\end{align*}
where $X_l\in R^{(m-l)\times(n-l)}$ and $Y_l\in R^{(n-l)\times n}$. At the start of the recursion we put, $X_0:=A^T$ and $Y_0:=\boldsymbol 1_n$. The $(l+1)$th condensation consists of the following procedure.

\noindent\textbf{Step 1 (preparation).} 
If $X_l=\boldsymbol 0$ or empty then stop; $\Rank(A^T)=l$ and $\ker_Q(A^T)=Y_l^TR^{n-l}$.  Let $k_{l+1}+1$ be the smallest column index of a nonzero column in $X_l$. This means that the first $k_{l+1}$ columns of $X$ are zero. Pick a row index $i$ with $x_{ik_{l+1}}\neq 0$ and replace  $X_l\mapsto X_l'$, $Y_l\mapsto Y_l'$, where $X_l'$ and $Y_l'$ are obtained from $X_l$ and $Y_l$, respectively, by applying a row permutation $\sigma$ to move the $i$th row to the first row. For clarity I put a box around the pivot $x_{ik_{l+1}}$.

\noindent\textbf{Step 2 (pc).} 
 The matrix $Z_l'':=[X_l'' |Y_l'']$  is obtained from $Z_l':=[X_l' |Y_l']$ as follows.
The entry of $Z_{l}''\in R^{( m-l-k_l-1) \times ( n+m-l-1)}$ at the $r$th row and $s$th column is 
\begin{align}
\label{eq:zij}
\begin{vmatrix}
z_{1\, k_{l+1}+1} & z_{1\, k_{l+1}+s}\\
z_{r+1\, k_{l+1}+1} & z_{r+1\, k_{l+1}+s}
\end{vmatrix}, 
\end{align}
where $Z_l'=:[z_{ij}]_{\substack{
i=1,\dotsc ,n-l\\
j=1,\dotsc ,2n-l}}$. The resulting pattern is $X_l'' |Y_l''$.

\noindent\textbf{Step 3 (cleaning up).}\footnote{This is optional up to the last iteration.} Let $L_i$ be the $i$th row of $Z_l'':=[X_l'' |Y_l'']$ and $c_i$ be the gcd of the entries of $L_i$ if $L_i$ is nonzero and put $c_i=1$  otherwise.  Let $Z_{l+1}:=[X_{l+1} |Y_{l+1}]$ be the matrix whose rows are $L_i/c_i$. 
The resulting pattern is $X_{l+1} |Y_{l+1}$, completing the $l$th condensation.

If $X_{m}$ is neither zero nor empty then $\ker_Q(A^T)$ is trivial. Let us look into two examples; more examples are elaborated in Section \ref{sec:CB}.

\begin{example}
Let $ R$ be the polynomial ring $\mathbb{Z}[ x,y,z]$ with quotient field $\mathbb{Q}( x,y,z)$ the field of rational functions in $x,y,z$ and coefficients in $\Q$. From the pc
\begin{align*}
&\begin{array}{ c c c|c c c c }
\boxed{1} & 1 & 1 & 1 & 0 & 0 & 0\\
x & y & z & 0 & 1 & 0 & 0\\
x^{2} & y^{2} & z^{2} & 0 & 0 & 1 & 0\\
x^{3} & y^{3} & z^{3} & 0 & 0 & 0 & 1
\end{array} \mapsto \begin{array}{ c c|c c c c }
\boxed{y-x} & z-x & -x & 1 & 0 & 0\\
y^{2} -x^{2} & z^{2} -x^{2} & -x^{2} & 0 & 1 & 0\\
y^{3} -x^{3} & z^{3} -x^{3} & -x^{3} & 0 & 0 & 1
\end{array}\\
&\mapsto \begin{array}{ c|c c c c }
( y-x)( z-x)( z-y) & ( y-x) xy & -\left( y^{2} -x^{2}\right) & y-x & 0\\
( y-x)( z-x)( z-y)( x+y+z) & \left( y^{2} -x^{2}\right) xy & -\left( y^{3} -x^{3}\right) & 0 & y-x
\end{array}\\
&=\begin{array}{ c|c c c c }
\boxed{( z-x)( z-y) }& xy & -y- x & 1 & 0\\
( z-x)( z-y)( x+y+z) & ( y+x) xy & -y^{2} -xy-x^{2} & 0 & 1
\end{array}\\
&\mapsto \begin{array}{ c|c c c c }
 & -( z-x)( z-y) xyz & ( z-x)( z-y)( xy+yz+xz) & -( z-x)( z-y)( x+y+z) & ( z-x)( z-y)
\end{array}\\
&=\begin{array}{ c|c c c c }
 & -xyz & xy+yz+xz & -x-y-z & 1
\end{array}
\end{align*}
we see that $\ker_{\mathbb{Q}( x,y,z)}\left(\begin{bmatrix}
1 & x & x^{2} & x^{3}\\
1 & y & y^{2} & y^{3}\\
1 & z & z^{2} & z^{3}
\end{bmatrix}\right) =\mathbb{Q}( x,y,z)\begin{bmatrix}
-xyz\\
xy+yz+xz\\
-x-y-z\\
1
\end{bmatrix} \subset \mathbb{Q}(x,y,z)^{4}$.
This is a special case of the identities $\sum _{i=0}^{n}( -x_{j})^{n-i} e_{i}( x_{1} ,\dotsc ,x_{n}) =0$ for elementary symmetric functions $e_{i}( x_{1} ,\dotsc ,x_{n})$, which can in turn easily be checked from the generating function $\prod _{i=1}^{n}( y-x_{i}) =\sum _{i=0}^{n}( -1)^{i} y^{n-i} e_{i}( x_{1} ,\dotsc ,x_{n})$.
I would like to point out some subtleties here. The ring $ R=\mathbb{Z}[ x,y,z]$ is not a PID and it is not guaranteed that the method produces a generating set for the whole kernel as an $R$-module. It turns out that for this it is sufficient to have a block of a permutation matrix in the matrix whose columns are our generators (more on this in Section \ref{sec:Rbases}). In the example at hand this is just the $1$ in the last entry of the generator.
\end{example}

\begin{example}
    To see that something can go wrong when $R$ is not a PID let us look into an example from physics. Let us consider $R=\mathbb R[ q_{1} ,q_{2} ,q_{3} ,p_{1} ,p_{2} ,p_{3}]$ and define 
    \begin{align*}
    \vec{J} =\vec{q} \times \vec{p} =\begin{bmatrix}
q_{2} p_{3} -q_{3} p_{2}\\
q_{3} p_{1} -q_{1} p_{3}\\
q_{1} p_{2} -q_{2} p_{1}
\end{bmatrix}.
\end{align*}
In physics the vectors $\vec q=\begin{bmatrix}
    q_1&q_2&q_3
\end{bmatrix}^T, \vec p=\begin{bmatrix}
    p_1&p_2&p_3
\end{bmatrix}^T$ and $\vec J$ are called position, momentum and angular momentum, respectively. Let us apply ker pc to the $1\times 3$-matrix $\vec{J}^{T}$:
\begin{align*}
\begin{array}{ c|c c c }
\boxed{J_{1}} & 1 & 0 & 0\\
J_{2} & 0 & 1 & 0\\
J_{3} & 0 & 0 & 1
\end{array} \mapsto \begin{array}{|c c c }
-J_{2} & J_{1} & 0\\
-J_{3} & 0 & J_{1}
\end{array}.
\end{align*}
The vectors $\vec v_1:=\begin{bmatrix}
-J_{2}\\
J_{1}\\
0
\end{bmatrix} ,\ \vec v_2:=\begin{bmatrix}
-J_{3}\\
0\\
J_{1}
\end{bmatrix} $ generate a rank two free submodule of $\ker_R\left(\vec{J}^{T}\right)$. But evidently $\vec{q}$ and $\vec{p}$ generate another rank two free submodule of $\ker_R\left(\vec{J}^{T}\right)$, and for degree reasons $\vec{q}$, $\vec{p}$ cannot be in $R\vec v_1+R\vec v_2$. People typically employ Gröbner bases for determining such kernels (see e.g. \cite{Greuel}). We will not discuss this topic here.
\end{example}

\begin{remark}
In this paper, the focus is on the case when $R=\Z$. Let me make some observations that arise when comparing the ker pc to Gaussian elimination of $A\in\Z^{n\times m}$:
\begin{enumerate}
    \item The ker pc of $[A^T|\boldsymbol 1_m]$, being equivalent to the first half of Gaussian elimination, requires only half as many steps as the Gaussian elimination of $A$ for determining $\ker_Q(A)$. On the other hand, $[A^T|\boldsymbol 1_m]$ is considerably bigger than $A$.
    \item In ker pc one can effortlessly avoid fractions\footnote{According to L.~Kronecker rational numbers are \emph{Menschenwerk} and require as such a lot of attention, i.e., brain capacity.}.
    \item  During pc the size of the matrices shrinks step-wise.
    \item When organized properly in situations such as chemical balancing the first steps of the ker pc are typically almost trivial and can be done quickly.
    \item The only minor drawback of the ker pc that I see is that for the determination of the image $A\Q^m$ one has to keep track of the permutations of the rows (see Subsection \ref{subsec:4pc}). This is irrelevant for chemical balancing.
    \item There is also a pc version of the Gaussian elimination of $A$, but the 2nd half of the process has more tricky signs and the notational protocols are less compact.
\end{enumerate}
\end{remark}

\label{subsec:ker}
\section{Consistency checks}
\label{subsec:hunt}

 A simple method to spot errors are check sums for rows\footnote{I learnt this from my dad.}. Here we put at the end of the pattern of ker  pc a double bar and another column that records the row sums. This column is included in the condensation procedure. If no error occurred the row sums have to match with the last column at each condensation step. For example,
\begin{align*}
&\begin{array}{ c c c|c c c c||c }
1 & 1 & 1 & 1 & 0 &0  & 0 & 4\\
1 & 2 & 3 & 0 & 1 & 0 & 0 & 7\\
1 & 3 & 6 & 0 & 0 & 1 & 0 & 11\\
1 & 4 & 10 & 0 & 0 & 0 & 1 & 16
\end{array} \mapsto \begin{array}{ c c|c c c c||c }
1 & 2 & -1 & 1 & 0 & 0 & 3\\
2 & 5 & -1 & 0 & 1 & 0 & 7\\
3 & 9 & -1 & 0 & 0 & 1 & 12
\end{array} \mapsto \begin{array}{ c|c c c c||c }
1 & 1 & -2 & 1 & 0 & 1\\
3 & 2 & -3 & 0 & 1 & 3
\end{array}\\
&\mapsto \begin{array}{ c|c c c c||c }
 & -1 & 3 & -3 & 1 & 0
\end{array}
\end{align*}
probably calculates $\ker_\Q\left( \begin{bmatrix}
1 & 1 & 1 & 1\\
1 & 2 & 3 & 4\\
1 & 3 & 6 & 10
\end{bmatrix}\right) =\mathbb{Q}\begin{bmatrix}
-1\\
3\\
-3\\
1
\end{bmatrix}$ correctly as the check sums match at each step.

If the errors persist one can also use check sums for columns. Here one has to be careful however: rescaling individual rows (cf. Step 3) messes up the check sums.

\section{Chemical balancing}
\label{sec:CB}
Before we take a closer look into mathematical stoichiometry (originating from \cite{Richter}) let me make a suggestion how to annotate the calculations when using pen and paper.
The matrices in chemical balancing are typically sparse, which means one has to write a lot of zeros. I will suppress them in the calculations and therefore it seems appropriate to use checkered paper. Alternatively, one can put dots instead of $0$ to keep track of the grid.

In chemical balancing one has given two finite sets: the set of atoms $I$ and the set of compounds $J$ together with an adjacency matrix $A=[ a_{ij}]_{i\in I,j\in J} \in \mathbb{Z}^{I\times J}$ whose integer entries are given by the rule that $a_{ij}$ counts the number of occurences of atom $i$ in compound $j$. A vector in the $\mathbb{Z}$-module 
\begin{align*}
    \ker_\Q( A)=\{\vec v=[v_j]_{j\in J}\in \mathbb Z^J\mid \forall i\in I: \ \sum_{j\in J} a_{ij}v_j=0\} 
\end{align*}
is symbolizing a mathematically possible chemical reaction involving the compounds of $J$. 
For practical matters it is better to work with generators of  the  $\ker_\Q( A)$ instead of the infinite set. If $\dim(\ker_\Q( A))=1$ our generator is automatically a generator for $\ker_\Z( A)$. For more information on how to determine the generators of $\ker_\Z( A)$ see Section \ref{sec:Rbases}.
The sets $I$ and $J$ come with an arbitrary chosen total order. The ker pc computations depend on the orders of $I$ and $J$, some orders being more convenient than others. In elimination the order on $I$ is merely used for presenting $A$ as a table. 
Typically, the entries of $A$ are from $\mathbb{N} =\{0,1,2,\dotsc \}$. 
If the compounds are ions it is convenient to deal also with negative entries (see Example \ref{ex:ionic} below).
If $\dim(\ker_\Q( A)) =1$ the generator of the $\mathbb{Z}$-module $\ker_\Z( A)$ is unique up to multiplication with a unit $\mathbb{Z}^{\times } =\{\pm 1\}$. The two options correspond to the two theoretically possible directions of the chemical reaction. Which one is more likely to be realized in nature is determined by the chemistry.

More generally, the absolute values of the entries of the generators of $\ker_\Z( A)$ are called the \emph{stoichiometric factors}. Our convention is that negative entries symbolize the reactants and positive entries symbolize reaction products. Mathematically, the choice of the global factor from $\mathbb{Z}^{\times } =\{\pm 1\}$ in front of each generator is arbitrary.

Let us now look into some examples. The general principle is to apply ker pc to the pattern $A^T|\boldsymbol 1$, the matrix $A$ being the adjacency matrix of the reaction. To avoid mistakes it makes sense to arrange the compounds into two groups: the reactants and the reaction products, if this grouping is known beforehand. Which one comes first can be decided according to which pivot is more simple. The ordering of the compounds within each group is also to be decided. I find it convenient to order the atoms lexicographically and view the compounds as monomials in  the set of atoms, forgetting about indices (i.e. all exponents in the monomial are $0$ or $1$). 
The monomials I order lexicographically if I do not see advantages in another ordering.\footnote{This does not decide if $\mathrm{CO}$ comes before $\mathrm{CO_2}$. If you want a definite rule take your favourite term order (see, e.g., \cite{Greuel, Yengui}), for example.}
The choice of ordering is of course only a matter of taste.

\begin{example}
Not every reaction is feasible. For example there is no reaction possible among the compounds $\mathrm{FeS}, \mathrm{H_{2}SO_{4}}, \mathrm{FeSO_{4}} ,\mathrm{H_{2} O}$. The reason is that the kernel of the adjacency matrix 
\begin{align*}
 \begin{array}{ c|c c c c }
 & \mathrm{FeS} & \mathrm{H_{2} SO_{4}} & \mathrm{FeSO_{4}} & \mathrm{H_{2} O}\\
\hline
\mathrm{Fe} & 1 &  & 1 & \\
\mathrm{H} &  & 2 &  & 2\\
\mathrm{O} &  & 4 & 4 & 1\\
\mathrm{S} & 1 & 1 & 1 & 
\end{array}
\end{align*}
is trivial. Let us check this using ker pc
\begin{align*}
&\begin{array}{ c c c c|c c c c }
\boxed{1} &  &  & 1 & 1 &  &  & \\
 & 2 & 4 & 1 &  & 1 &  & \\
1 &  & 4 & 1 &  &  & 1 & \\
 & 2 & 1 &  &  &  &  & 1
\end{array} \mapsto \begin{array}{ c c c|c c c c }
\boxed{2} & 4 & 1 &  & 1 &  & \\
 & 4 &  & -1 &  & 1 & \\
2 & 1 &  &  &  &  & 1
\end{array} \mapsto \begin{array}{ c c|c c c c }
\boxed{8} &  & -2 &  & 2 & \\
-6 & -2 &  & -2 &  & 2
\end{array}\\
&\mapsto \begin{array}{ c|c c c c }
\underline{-16} & -12 & -16 & 12 & 16
\end{array}.
\end{align*}
There are not enough compounds in the list. Maybe one should add some rust.
\end{example}

\begin{example}
\label{ex:rost}
There can also be too many compounds. Then the balanced reaction is not unique up to sign. To see an example with two generators let us look into the adjacency matrix 
\begin{align*}
\begin{array}{ c|c c c c }
 & \mathrm{Fe} & \mathrm{O_{2}} & \mathrm{FeO} & \mathrm{Fe_{2} O_{3}}\\
\hline
\mathrm{Fe} & 1 &  & 1 & 2\\
\mathrm{O} &  & 2 & 1 & 3
\end{array}.
\end{align*}
The pivotal condensation gives
\begin{align*}
\begin{array}{ c c|c c c c }
\boxed{1} &  & 1 &  &  & \\
 & 2 &  & 1 &  & \\
1 & 1 &  &  & 1 & \\
2 & 3 &  &  &  & 1
\end{array} \mapsto \begin{array}{ c|c c c c }
\boxed{2} &  & 1 &  & \\
1 & -1 &  & 1 & \\
3 & -2 &  &  & 1
\end{array} \mapsto \begin{array}{|c c c c }
-2 & -1 & 2 & \\
-4 & -3 &  & 2
\end{array}.
\end{align*}
The reactions corresponding to the generators above are
$2\mathrm{Fe+O_{2}\rightarrow 2FeO}$ and $4\mathrm{Fe+3O_{2}\rightarrow 2Fe_{2} O_{3}}$. 

There is a catch here, however. The reader may readily check that the linear independent vectors 
\begin{align*}    
\begin{bmatrix}
-2\\
-1\\
2\\
0
\end{bmatrix}, \ \ \ \begin{bmatrix}
1\\
0\\
-3\\
1
\end{bmatrix}
\end{align*}
are in the kernel of the adjacency matrix. On the other hand, 
\begin{align*}\begin{bmatrix}
-2 & 1\\
-1 & 0\\
2 & -3\\
0 & 1
\end{bmatrix} =\begin{bmatrix}
-2 & -4\\
-1 & -3\\
2 & 0\\
0 & 2
\end{bmatrix}\begin{bmatrix}
1 & -3/2\\
0 & 1/2\\
\end{bmatrix}\end{align*}
is not representing an integer combination of the generators obtained from ker pc. 
Accordingly, a more fundamental set of generators for the chemical reactions is actually $2\mathrm{Fe+O_{2}\rightarrow 2FeO}$ and $3\mathrm{Fe0}\rightarrow \mathrm{Fe_{2} O_{3}+Fe}$. These resulting generators are unique up the action
of $\operatorname{GL_{2}}(\mathbb{Z})$.

In fact, the problem is visible in the gcd $\Delta _{2}$ of the $2\times 2$-minors of $X:=\begin{bmatrix}
-2 & -4\\
-1 & -3\\
2 & 0\\
0 & 2
\end{bmatrix}$, which is $2$. The image $\operatorname{im}_\Z( X) =X\mathbb{Z}^{2}$ of $X$ cannot be saturated (see Theorem \ref{thm:characterization}), and the columns of this matrix cannot form an $\mathbb{Z}$-basis for the kernel of $A$. If $R$ is a PID, $A\in R^{n\times m}$ is of rank $p$, $AX=0$ and $X$ has rank $n-p$, then $\ker_R( A) =\operatorname{im}_R( X)$ if and only if the gcd $\Delta _{n-p}$ of the $( n-p) \times ( n-p)$-minors is a unit. Sometimes ker pc produces generators whose $\Delta _{n-p}$ is not a unit. For details on how one can repair this see Section \ref{sec:Rbases}.

\end{example}

\begin{example} \label{ex:Fischer-Tropsch}
Out of a head of a chemist come mostly adjacency matrices that have just one chemical equation. For example,
the Fischer-Tropsch reaction for the synthesis of paraffin involves the following adjacency matrix 
\begin{align}\label{eq:FT}
\begin{array}{ c|c c c c }
 & \mathrm{CO} & \mathrm{H_{2}} & \mathrm{C_{n} H_{2n+2}} & \mathrm{H_{2} O}\\
\hline
\mathrm{C} & 1 &  & n & \\
\mathrm{H} &  & 2 & 2n+2 & 2\\
\mathrm{O} & 1 &  &  & 1
\end{array}.
\end{align}
We work symbolically, i.e., over the ring $\mathbb{Z}[ n]$. The ker pc is
\begin{align*}
&\begin{array}{ c c c|c c c c }
\boxed{1} &  & 1 & 1 &  &  & \\
 & 2 &  &  & 1 &  & \\
n & 2n+2 &  &  &  & 1 & \\
 & 2 & 1 &  &  &  & 1
\end{array} \mapsto \begin{array}{ c c|c c c c }
\boxed{2} &  &  & 1 &  & \\
2n+2 & -n & -n &  & 1 & \\
2 & 1 &  &  &  & 1
\end{array} 
\mapsto 
\begin{array}{ c|c c c c }
-2n & -2n & -2n-2 & 2 & \\
2 &  & -2 &  & 2
\end{array}\\
&=\begin{array}{ c|c c c c }
\boxed{n} & n & n+1 & -1 & \\
1 &  & -1 &  & 1
\end{array} \mapsto \begin{array}{ c|c c c c }
 & -n & -2n-1 & 1 & n
\end{array}
\end{align*}
and hence $n\mathrm{CO} +( 2n+1)\mathrm{H_{2}}\rightarrow \mathrm{C_{n} H_{2n+2}} +n\mathrm{H_{2} O}$. As $\begin{bmatrix}
 0& -n & -2n-1 & 1 & n
\end{bmatrix}^T$ has a unit entry it actually generates the kernel as a $\Z[n]$-module (see Section \ref{sec:Rbases})\footnote{The ring $\Z[n]$ is not a PID and hence Theorem \ref{thm:characterization} cannot be applied.}.
We will come back to this example in Section \ref{sec:quivered}.
\end{example}

\begin{example}
\label{ex:ionic}
In ionic reactions I treat the electric charge as an atom, as in the following example:
\begin{align*}
\begin{array}{ c|c c c c c c }
 & \mathrm{H_{2} O} & \mathrm{MnO_{4}^{-}} & \mathrm{SO_{3}^{2-}} & \mathrm{OH^{-}} & \mathrm{MnO_{2}} & \mathrm{SO_{4}^{2-}}\\
\hline
\mathrm{H} & 2 &  &  & 1 &  & \\
\mathrm{Mn} &  & 1 &  &  & 1 & \\
\mathrm{S} &  &  & 1 &  &  & 1\\
\mathrm{O} & 1 & 4 & 3 & 1 & 2 & 4\\
\mbox{charge} &  & -1 & -2 & -1 &  & -2
\end{array}
\end{align*}
If during ker pc in a row under the pivot there is a zero the corresponding row (with first zero deleted) is simply multiplied with the pivot.
This in turn is then divided out. I take the liberty to combine the two steps and simply copy over the row, omitting the first entry. This occurs in the first step of the following ker pc:
\begin{align*}
&\begin{array}{ c c c c c|c c c c c c }
\boxed{2} &  &  & 1 &  & 1 &  &  &  &  & \\
 & 1 &  & 4 & -1 &  & 1 &  &  &  & \\
 &  & 1 & 3 & -2 &  &  & 1 &  &  & \\
1 &  &  & 1 & -1 &  &  &  & 1 &  & \\
 & 1 &  & 2 &  &  &  &  &  & 1 & \\
 &  & 1 & 4 & -2 &  &  &  &  &  & 1
\end{array} \mapsto \begin{array}{ c c c c|c c c c c c }
\boxed{1} &  & 4 & -1 &  & 1 &  &  &  & \\
 & 1 & 3 & -2 &  &  & 1 &  &  & \\
 &  & 1 & -2 & -1 &  &  & 2 &  & \\
1 &  & 2 &  &  &  &  &  & 1 & \\
 & 1 & 4 & -2 &  &  &  &  &  & 1
\end{array}\\
&\mapsto \begin{array}{ c c c|c c c c c c }
\boxed{1} & 3 & -2 &  &  & 1 &  &  & \\
 & 1 & -2 & -1 &  &  & 2 &  & \\
 & -2 & 1 &  & -1 &  &  & 1 & \\
1 & 4 & -2 &  &  &  &  &  & 1
\end{array} \mapsto \begin{array}{ c c|c c c c c c }
\boxed{1} & -2 & -1 &  &  & 2 &  & \\
-2 & 1 &  & -1 &  &  & 1 & \\
1 &  &  &  & -1 &  &  & 1
\end{array}\\
&\mapsto \begin{array}{ c|c c c c c c }
\boxed{-3} & -2 & -1 &  & 4 & 1 & \\
2 & 1 &  & -1 & -2 &  & 1
\end{array} \mapsto \begin{array}{ c|c c c c c c }
 & 1 & 2 & 3 & -2 & -2 & -3
\end{array} \mapsto \begin{array}{ c|c c c c c c }
 & -1 & -2 & -3 & 2 & 2 & 3
\end{array}.
\end{align*}
The balanced reaction is $\mathrm{H_{2} O} +2\mathrm{MnO_{4}^{-}} +3\mathrm{SO_{3}^{2-}}\rightarrow 2\mathrm{OH^{-}} +2\mathrm{MnO_{2}} +3\mathrm{SO_{4}^{2-}}$.
\end{example}

\begin{example}
An important related question in the physical sciences is if it is possible to form a rational dimensionless quantity out of a finite set of physical quantities. Such a dimensionless quantity serves as a characteristic number of the corresponding physical system and can be fed into transcendental functions, such as $\exp$. For example, we may ask: Can we make a dimensionless quantity out of viscosity $\eta $, density $\rho $, velocity $v$ and length $l$? Now physical quantities are compounds of the fundamental units (the analogue of the atoms above). We can form an adjacency matrix, which can also have negative entries: 
\begin{align*}
\begin{array}{ c|c c c c }
 & [ \eta ] =kg\ m^{-1} s^{-1} & [ \rho ] =kg\ m^{-3} & [ v] =m\ s^{-1} & [ l] =m\\
\hline
kg & 1 & 1 &  & \\
m & -1 & -3 & 1 & 1\\
s & -1 &  & -1 & 
\end{array}
\end{align*}
The ker pc produces
\begin{align*}
&\begin{array}{ c c c|c c c c }
\boxed{1} & -1 & -1 & 1 &  &  & \\
1 & -3 &  &  & 1 &  & \\
 & 1 & -1 &  &  & 1 & \\
 & 1 &  &  &  &  & 1
\end{array} \mapsto \begin{array}{ c c|c c c c }
\boxed{-2} & 1 & -1 & 1 &  & \\
1 & -1 &  &  & 1 & \\
1 &  &  &  &  & 1
\end{array} \mapsto \begin{array}{ c|c c c c }
\boxed{1} & 1 & -1 & -2 & \\
-1 & 1 & -1 &  & -2
\end{array} \\
&\mapsto \begin{array}{ c|c c c c }
 & 2 & -2 & -2 & -2
\end{array}\mapsto \begin{array}{ c|c c c c }
 & -1 & 1 & 1 & 1
\end{array}.
\end{align*}
This vector of exponents appears in the famous Reynolds number $\mathrm{Re} =\eta ^{-1} \rho vl$. \ 
\end{example}

\begin{example} I am taking an example from \cite{Stout} to illustrate that the ker pc method can be used to balance chemical equations that are regarded to be challenging.
In the following we use the shorthand $\mathrm{Cr_{7} N_{66} H_{96} C_{42} O_{24}}$ for the complex $\mathrm{[ Cr( N_{2} H_{4} CO)_{6}]_{4}[ Cr( CN)_{6}]_{3}}$. We would like to analyze the adjacency matrix 
\begin{align*}
\begin{array}{ c|c c c c c c c c c }
 & \mathrm{CO_{2}} & \mathrm{KMnO_{4}} & \mathrm{MnSO_{4}} & \mathrm{K_{2} Cr_{2} O_{7}} & \mathrm{KNO_{3}} & \mathrm{H_{2} SO_{4}} & \mathrm{K_{2} SO_{4}} & \mathrm{H_{2} O} & \mathrm{Cr_{7} N_{66} H_{96} C_{42} O_{24}}\\
\hline
\mathrm{C} & 1 &  &  &  &  &  &  &  & 42\\
\mathrm{Mn} &  & 1 & 1 &  &  &  &  &  & \\
\mathrm{Cr} &  &  &  & 2 &  &  &  &  & 7\\
\mathrm{N} &  &  &  &  & 1 &  &  &  & 66\\
\mathrm{S} &  &  & 1 &  &  & 1 & 1 &  & \\
\mathrm{H} &  &  &  &  &  & 2 &  & 2 & 96\\
\mathrm{K} &  & 1 &  & 2 & 1 &  & 2 &  & \\
\mathrm{O} & 2 & 4 & 4 & 7 & 3 & 4 & 4 & 1 & 24
\end{array}
\end{align*}
Here I choose orderings of $I$ and $J$ that seem adequate for keeping the computation as simple as possible (see the remark at the end of Section \ref{sec:quivered}). Using the policy explained in Example \ref{ex:ionic} the first steps involve mostly copying and pasting. 

\begin{align*}
& \begin{array}{ c c c c c c c c|c c c c c c c c c }
\boxed{1} &  &  &  &  &  &  & 2 & 1 &  &  &  &  &  &  &  & \\
 & 1 &  &  &  &  & 1 & 4 &  & 1 &  &  &  &  &  &  & \\
 & 1 &  &  & 1 &  &  & 4 &  &  & 1 &  &  &  &  &  & \\
 &  & 2 &  &  &  & 2 & 7 &  &  &  & 1 &  &  &  &  & \\
 &  &  & 1 &  &  & 1 & 3 &  &  &  &  & 1 &  &  &  & \\
 &  &  &  & 1 & 2 &  & 4 &  &  &  &  &  & 1 &  &  & \\
 &  &  &  & 1 &  & 2 & 4 &  &  &  &  &  &  & 1 &  & \\
 &  &  &  &  & 2 &  & 1 &  &  &  &  &  &  &  & 1 & \\
42 &  & 7 & 66 &  & 96 &  & 24 &  &  &  &  &  &  &  &  & 1
\end{array}\\
& \mapsto \begin{array}{ c c c c c c c|c c c c c c c c c }
\boxed{1} &  &  &  &  & 1 & 4 &  & 1 &  &  &  &  &  &  & \\
1 &  &  & 1 &  &  & 4 &  &  & 1 &  &  &  &  &  & \\
 & 2 &  &  &  & 2 & 7 &  &  &  & 1 &  &  &  &  & \\
 &  & 1 &  &  & 1 & 3 &  &  &  &  & 1 &  &  &  & \\
 &  &  & 1 & 2 &  & 4 &  &  &  &  &  & 1 &  &  & \\
 &  &  & 1 &  & 2 & 4 &  &  &  &  &  &  & 1 &  & \\
 &  &  &  & 2 &  & 1 &  &  &  &  &  &  &  & 1 & \\
 & 7 & 66 &  & 96 &  & -60 & -42 &  &  &  &  &  &  &  & 1
\end{array}\\
&\mapsto \begin{array}{ c c c c c c|c c c c c c c c c }
 &  & 1 &  & -1 &  &  & -1 & 1 &  &  &  &  &  & \\
2 &  &  &  & 2 & 7 &  &  &  & 1 &  &  &  &  & \\
 & 1 &  &  & 1 & 3 &  &  &  &  & 1 &  &  &  & \\
 &  & 1 & 2 &  & 4 &  &  &  &  &  & 1 &  &  & \\
 &  & 1 &  & 2 & 4 &  &  &  &  &  &  & 1 &  & \\
 &  &  & 2 &  & 1 &  &  &  &  &  &  &  & 1 & \\
7 & 66 &  & 96 &  & -60 & -42 &  &  &  &  &  &  &  & 1
\end{array}\\
&\mapsto \begin{array}{ c c c c c c|c c c c c c c c c }
\boxed{2} &  &  &  & 2 & 7 &  &  &  & 1 &  &  &  &  & \\
 & 1 &  &  & 1 & 3 &  &  &  &  & 1 &  &  &  & \\
 &  & 1 & 2 &  & 4 &  &  &  &  &  & 1 &  &  & \\
 &  & 1 &  & 2 & 4 &  &  &  &  &  &  & 1 &  & \\
 &  &  & 2 &  & 1 &  &  &  &  &  &  &  & 1 & \\
 &  & 1 &  & -1 &  &  & -1 & 1 &  &  &  &  &  & \\
7 & 66 &  & 96 &  & -60 & -42 &  &  &  &  &  &  &  & 1
\end{array}\\
&\mapsto \begin{array}{ c c c c c|c c c c c c c c c }
\boxed{1} &  &  & 1 & 3 &  &  &  &  & 1 &  &  &  & \\
 & 1 & 2 &  & 4 &  &  &  &  &  & 1 &  &  & \\
 & 1 &  & 2 & 4 &  &  &  &  &  &  & 1 &  & \\
 &  & 2 &  & 1 &  &  &  &  &  &  &  & 1 & \\
 & 1 &  & -1 &  &  & -1 & 1 &  &  &  &  &  & \\
132 &  & 192 & -14 & -169 & -84 &  &  & -7 &  &  &  &  & 2
\end{array}\\
&\mapsto \begin{array}{ c c c c|c c c c c c c c c }
\boxed{1} & 2 &  & 4 &  &  &  &  &  & 1 &  &  & \\
1 &  & 2 & 4 &  &  &  &  &  &  & 1 &  & \\
 & 2 &  & 1 &  &  &  &  &  &  &  & 1 & \\
1 &  & -1 &  &  & -1 & 1 &  &  &  &  &  & \\
 & 192 & -146 & -565 & -84 &  &  & -7 & -132 &  &  &  & 2
\end{array}
\\
&\mapsto \begin{array}{ c c c|c c c c c c c c c }
\boxed{-2} & 2 &  &  &  &  &  &  & -1 & 1 &  & \\
2 &  & 1 &  &  &  &  &  &  &  & 1 & \\
-2 & -1 & -4 &  & -1 & 1 &  &  & -1 &  &  & \\
192 & -146 & -565 & -84 &  &  & -7 & -132 &  &  &  & 2
\end{array}\\
&\mapsto \begin{array}{ c c|c c c c c c c c c }
-4 & -2 &  &  &  &  &  & 2 & -2 & -2 & \\
6 & 8 &  & 2 & -2 &  &  &  & 2 &  & \\
-92 & 1130 & 168 &  &  & 14 & 264 & 192 & -192 &  & -4
\end{array}\\
\end{align*}
\begin{align*}
&\mapsto \begin{array}{ c c|c c c c c c c c c }
\boxed{-2} & -1 &  &  &  &  &  & 1 & -1 & -1 & \\
3 & 4 &  & 1 & -1 &  &  &  & 1 &  & \\
-46 & 565 & 84 &  &  & 7 & 132 & 96 & -96 &  & -2
\end{array}\\
&\mapsto \begin{array}{ c|c c c c c c c c c }
-5 &  & -2 & 2 &  &  & -3 & 1 & 3 & \\
-1176 & -168 &  &  & -14 & -264 & -146 & 146 & -46 & 4
\end{array}\\
&\mapsto \begin{array}{ c|c c c c c c c c c }
\boxed{-5} &  & -2 & 2 &  &  & -3 & 1 & 3 & \\
-588 & -84 &  &  & -7 & 132 & -73 & 73 & -23 & 2
\end{array}\\
& \mapsto \begin{array}{ c|c c c c c c c c c }
 & 420 & -1176 & 1176 & 35 & 660 & -1399 & 223 & 1879 & -10
\end{array}
\end{align*}
The balanced reaction is\footnote{I now deserve a grade A from Professor Stout (Central College, Pella, IA) for the entire semester!}
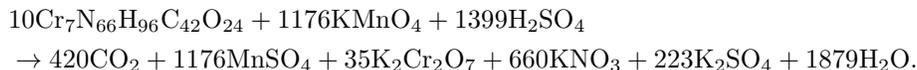
\begin{align*}
&10\mathrm{Cr_{7} N_{66} H_{96} C_{42} O_{24}} +1176\mathrm{KMnO_{4}} +1399\mathrm{H_{2} SO_{4}}\\
&\rightarrow 420\mathrm{CO_{2}} +1176\mathrm{MnSO_{4}} +35\mathrm{K_{2} Cr_{2} O_{7}} +660\mathrm{KNO_{3}} +223\mathrm{K_{2} SO_{4}} +1879\mathrm{H_{2} O}.
\end{align*}
As $4$ of $8$ atoms in this example occur only in $2$ compounds respectively it would also make sense to apply the quivering method of the next section.
\end{example}

\section{Quivered ker pc}
\label{sec:quivered}
When balancing reactions chemists often exploit substitutions arising from all the atoms that only occur in just two compounds. In this way it is often possible 
to avoid the formalism as the resulting linear systems are quite simple. I am trying here to elaborate the chemist's method as an algorithmic tool to preprocess the linear system of the reaction. The preprocessing is referred to as \emph{quivering} as there are a quivers (i.e. directed graphs) involved. The difficulty is to find all of the obvious reasons to put variables to zero.

We are aiming to solve the system 
\begin{align}\label{eq:ls}
\sum _{i\in I} a_{ij} v_{j} =0
\end{align} for all $j\in J$ with $A=[ a_{ij}] \in R^{I\times J}$. Here $I$ and $J$ are finite sets and $I,J$ are understood to carry a total order.
For the quivering only the order on $J$ is relevant.
The plan is to construct minimal subsets $\hat{I} \subseteq I,\ \hat{J} \subseteq J$ and the quivered linear system $\sum _{i\in \hat{I}}\hat{a}_{ij} v_{j} =0$ for all $j\in \hat{J}$ such that all solutions of \eqref{eq:ls} arise from solutions of the quivered system by substitutions.

To this end I am introducing a process referred to as \emph{pruning} of the subsets $I'\subseteq I$ and $J'\subseteq J$.
For each integer $n\geq 0$ we put
\begin{align*}
I'_{n} :=\{i\in I'\mid n=\left|\{j\in J'\mid a_{ij} \neq 0\}\right|\}.
\end{align*}
For each $i\in I'_{1}$ there is a unique $j\in J'$ such that $a_{ij} \neq 0$ this defines a map $\phi':I'\to J'$. We put $I'':=I'\backslash(I_0'\cup I_1')$ and $J'':=J'\backslash \phi'(I')$. Now $I''_2:=\{i\in I''\mid 2=\left|\{j\in J''\mid a_{ij} \neq 0\}\right|\}$ is a quiver. In fact, we put  $s(i)$ for $i\in I''_2$ to be the smaller of the two $j$'s such that $a_{ij} \neq 0$, the other one is called $t(i)$. This provides source and target the maps $s,t:I''_2\to J''$. 
A path $\gamma=(i_1,\dots i_{|\gamma|})$ of length $|\gamma|$ in $I''_2$ is a sequence of edges $i_1,\dots i_{|\gamma|})\in I''_2$ such that for each $p=1,\dots,|\gamma|-1$ we have $t(i_p)=s(i_{p+1})$.
Two paths $\gamma=(i_1,\dots ,i_{|\gamma|})$ and $\gamma'=(i'_1,\dots ,i'_{|\gamma'|})$ in $I''_2$ are called \emph{equivalent}, $\gamma\sim\gamma'$,  if
\begin{enumerate}
    \item $s(i_1)=s(i'_1)$ and $t(i_{|\gamma|})=s(i'_{|\gamma'|})$,
    \item and 
    \begin{align}\label{eq: consistent}
        ( -1)^{|\gamma |}\prod _{p=1}^{|\gamma |}\frac{a_{i_{p} s( i_{p})}}{a_{i_{p} t( i_{p})}} =( -1)^{|\gamma '|}\prod _{q=1}^{|\gamma '|}\frac{a_{i'_{q} s( i'_{q})}}{a_{i'_{q} t( i'_{q})}}.
    \end{align}
\end{enumerate}
If (1) holds and (2) is violated they are called \emph{inconsistent}
, $\gamma\perp\gamma'$. The set of problematic edges is
\begin{align}
    I_2''^!:=\{i\mid i\mbox{ an edge of }\gamma \mbox{ or } \gamma', \  \gamma\perp\gamma'\}.
\end{align}
In the end by Eqn. \eqref{eq:ls} all the $v_j$ for $j\in s(I_2''^!)\cup t(I_2''^!)\cup \phi'(I')$ have to be put zero.
The pruned sets are $\mathrm{pr}(I'):=I_2''\backslash I_2''^!$ and $\mathrm{pr}(J'):=J''\backslash ( s(I_2''^!)\cup t(I_2''^!))$. 
The quiver $s,t: (\mathrm{pr}(I'))_2\to \mathrm{pr}(J)$ has no oriented cycles, since an edge increases the order by construction. All paths in $s,t: (\mathrm{pr}(I'))_2\to \mathrm{pr}(J')$ with the same initial and terminal vertex are equivalent. 
Unfortunately, the pruning should be iterated in general, to get an optimal result. But starting with $I$, $J$ there must be a smallest $k$ such that $\mathrm{pr}^{k+1}(I)=\mathrm{pr}^{k}(I)$ and $\mathrm{pr}^{k+1}(J)=\mathrm{pr}^{k}(J)$; it is referred to as the \emph{pruning depth} $k$.

Let $\gamma$ be a maximal path (possibly trivial) in the quiver $s,t: (\mathrm{pr}^{k}(I))_2\to \mathrm{pr}^{k}(J)$. 
Let us denote by $P$ a set of representatives of equivalence classes of such paths. The quivered system is obtained as follows. We put
$\hat{I} :=\operatorname{pr}^{k}( I) ,\ \hat{J} :=\left\{\operatorname{in}( \gamma ) \mid \gamma \in P\right\}$ and define for $i\in \hat{I} ,j\in \hat{J}$:
\begin{align}
     b_{ij} :=a_{ij} +\sum _{\gamma =(i_{1} ,i_{2} ,\dotsc ,i_{|\gamma |}) \in P: |\gamma|\geq 1, \operatorname{in}(\gamma)=j}( -1)^{|\gamma |} a_{it(i_{|\gamma |})}\prod _{p=1}^{|\gamma |}\frac{a_{i_{p} s( i_{p})}}{a_{i_{p} t( i_{p})}} \in Q.
\end{align}
We get a matrix $B=[ b_{ij}] \in Q^{\hat{I} \times \hat{J}}$. We multiply the rows in $B$ appropriately with scalars in $R$ to obtain a matrix $\hat{A} \in R^{\hat{I} \times \hat{J}}$. 

To reconstruct a solution of \eqref{eq:ls} we proceed as follows. 
By ker pc we find $\ker_Q\hat{A} =\hat{W} R^m$ for some matrix $\hat{W} =[w_{j\mu}] \in R^{\hat{J}\times m}$ with $m:=|\hat{J} |-\Rank(\hat{A}) =|J|-\Rank( A)$. We put for $\mu=\{1,\dots,m\}$
\begin{align}\label{eq:reconstr}
w_{j\mu} :=\begin{cases}
0 & \mbox{if }j\in J\backslash\operatorname{pr}^k(J) \\
w_{\operatorname{in}( \gamma )\mu}( -1)^{|\gamma |}\prod _{p=1}^{|\gamma |}\frac{a_{i_{p} s( i_{p})}}{a_{i_{p} t( i_{p})}} & \mbox{if there is a } \gamma \in P( j) \ 
\end{cases}
\end{align}
to obtain $W=[w_{j\mu}]\in Q^{J\times m}$. We multiply appropriately columns of $W$ with elements of $R$ to arrive at $V\in R^{J}$ with coprime columns (this is always possible if $R$ is a gcd domain). The columns of $V$ form a basis of the $R$-module $\ker_R A$.
\begin{example}
We look once again into Example \ref{ex:Fischer-Tropsch}.
The pruning depth of the Fischer-Tropsch adjacency matrix \eqref{eq:FT} is $0$. Reactions that make sense chemically should not need pruning. We have $I_0=I_1=\emptyset$ and the quiver $s,t:I_2\to J$ 
\begin{align*}
    \mathrm{H_{2} O\xleftarrow{\ \ \ O\ \ \ } CO\xrightarrow{\ \ \ C\ \ \ } C_{n} H_{2n+2}}.
\end{align*}
The quivered linear system is given by $\hat{J}=\{\mathrm{CO,H_2}\}$ and $\hat{I}=\{\mathrm{H}\}$ with $\hat{A}=[-2n-1,n]$ with quivered solution $[-n,-2n-1]^T$. Using Eqn. \eqref{eq:reconstr} we reproduce the result of Example \ref{ex:Fischer-Tropsch}.
\end{example}

\begin{example}
To avoid abuse of terrestrial chemistry (it seems important, see \cite{Forester,whynots,IceB}) I take a look at an adjacency matrix proposed by a chemistry student from Betelgeuse:
\begin{align*}
   \begin{array}{ c|c c c c c c c c }
 & ACD & ABDE & B_{2} C_{3} DE & DEFGH & E_{2} H & E_{6} F & FG_{3} & G H\\
\hline
A & 1 & 1 &  &  &  &  &  & \\
B &  & 1 & 2 &  &  &  &  & \\
C & 1 &  & 3 &  &  &  &  & \\
D & 1 & 1 & 1 & 1 &  &  &  & \\
E &  & 1 & 1 & 1 & 2 & 6 &  & \\
F &  &  &  & 1 &  & 1 & 1 & \\
G &  &  &  & 1 &  &  & 3 & 1\\
H &  &  &  & 1 & 1 &  &  & 1
\end{array}.
\end{align*}
Note that $I_{0} =I_{1} =\emptyset $ and the two paths 
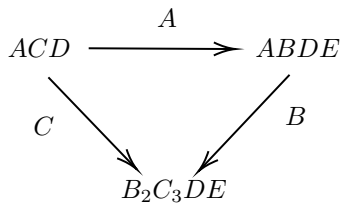
\begin{center}
\tikzset{every picture/.style={line width=0.75pt}} 
\begin{tikzpicture}[x=0.75pt,y=0.75pt,yscale=-1,xscale=1]

\draw    (67.33,37.33) -- (139.38,37.71) ;
\draw [shift={(141.38,37.72)}, rotate = 180.3] [color={rgb, 255:red, 0; green, 0; blue, 0 }  ][line width=0.75]    (10.93,-3.29) .. controls (6.95,-1.4) and (3.31,-0.3) .. (0,0) .. controls (3.31,0.3) and (6.95,1.4) .. (10.93,3.29)   ;
\draw    (168.63,50.73) -- (124.01,98.02) ;
\draw [shift={(122.63,99.47)}, rotate = 313.34] [color={rgb, 255:red, 0; green, 0; blue, 0 }  ][line width=0.75]    (10.93,-3.29) .. controls (6.95,-1.4) and (3.31,-0.3) .. (0,0) .. controls (3.31,0.3) and (6.95,1.4) .. (10.93,3.29)   ;
\draw    (47,52) -- (90.25,97.03) ;
\draw [shift={(91.63,98.47)}, rotate = 226.16] [color={rgb, 255:red, 0; green, 0; blue, 0 }  ][line width=0.75]    (10.93,-3.29) .. controls (6.95,-1.4) and (3.31,-0.3) .. (0,0) .. controls (3.31,0.3) and (6.95,1.4) .. (10.93,3.29)   ;

\draw (25,31) node [anchor=north west][inner sep=0.75pt]    {$ACD$};
\draw (165,65.4) node [anchor=north west][inner sep=0.75pt]    {$B$};
\draw (82,102.4) node [anchor=north west][inner sep=0.75pt]    {$B_{2} C_{3} DE$};
\draw (100.33,16) node [anchor=north west][inner sep=0.75pt]    {$A$};
\draw (150,31) node [anchor=north west][inner sep=0.75pt]    {$ABDE$};
\draw (38,70.4) node [anchor=north west][inner sep=0.75pt]    {$C$};

\end{tikzpicture}
\end{center}
\noindent are inconsistent and $v_{ACD}=v_{B_2C_3DE}=v_{ABDE}=0$. Hence 
\begin{align*}
\operatorname{pr}( I) =\{D,E,F,G\} \mbox{ and }\operatorname{pr}( J) =\{DEFGH,E_{2} G,E_{6} F,FG_{3} ,G H\}.     
\end{align*}
While $\operatorname{pr}( I)_{0} =\emptyset $ we have $\operatorname{pr}( I)_{1} =\{D\}$ and $v_{DEFGH} =0$. On the other hand, the two paths
\begin{center}
\tikzset{every picture/.style={line width=0.75pt}} 
\begin{tikzpicture}[x=0.75pt,y=0.75pt,yscale=-1,xscale=1]

\draw    (67.33,32.33) -- (139.38,32.71) ;
\draw [shift={(141.38,32.72)}, rotate = 180.3] [color={rgb, 255:red, 0; green, 0; blue, 0 }  ][line width=0.75]    (10.93,-3.29) .. controls (6.95,-1.4) and (3.31,-0.3) .. (0,0) .. controls (3.31,0.3) and (6.95,1.4) .. (10.93,3.29)   ;
\draw    (174.63,42.73) -- (175.11,91.02) ;
\draw [shift={(175.13,93.02)}, rotate = 269.43] [color={rgb, 255:red, 0; green, 0; blue, 0 }  ][line width=0.75]    (10.93,-3.29) .. controls (6.95,-1.4) and (3.31,-0.3) .. (0,0) .. controls (3.31,0.3) and (6.95,1.4) .. (10.93,3.29)   ;
\draw    (41,43) -- (42.09,92.02) ;
\draw [shift={(42.13,94.02)}, rotate = 268.73] [color={rgb, 255:red, 0; green, 0; blue, 0 }  ][line width=0.75]    (10.93,-3.29) .. controls (6.95,-1.4) and (3.31,-0.3) .. (0,0) .. controls (3.31,0.3) and (6.95,1.4) .. (10.93,3.29)   ;
\draw    (80,108.99) -- (143.13,109.02) ;
\draw [shift={(78,108.98)}, rotate = 0.03] [color={rgb, 255:red, 0; green, 0; blue, 0 }  ][line width=0.75]    (10.93,-3.29) .. controls (6.95,-1.4) and (3.31,-0.3) .. (0,0) .. controls (3.31,0.3) and (6.95,1.4) .. (10.93,3.29)   ;

\draw (20,25) node [anchor=north west][inner sep=0.75pt]    {$E_{2} H$};
\draw (182,56.4) node [anchor=north west][inner sep=0.75pt]    {$F$};
\draw (23,101.4) node [anchor=north west][inner sep=0.75pt]    {$G H$};
\draw (100.33,12) node [anchor=north west][inner sep=0.75pt]    {$E$};
\draw (157,25) node [anchor=north west][inner sep=0.75pt]    {$E_{6} F$};
\draw (21,61.4) node [anchor=north west][inner sep=0.75pt]    {$H$};
\draw (156,100.38) node [anchor=north west][inner sep=0.75pt]    {$FG_{3}$};
\draw (103,116.38) node [anchor=north west][inner sep=0.75pt]    {$G$};

\end{tikzpicture}
\end{center}

\noindent are equivalent. The pruning depth is $2$. The diagram represents the quiver $\operatorname{pr}^{2}( I)_{2}\rightarrow \operatorname{pr}^{2}( J)$.
We have $\hat{I} =\emptyset $, $\hat{J} =\{E_{2} H\}$ and the quivered linear system is empty. The solution of the quivered system is without loss of generality $w_{E_{2} H} =1$. The coprime integer solution is $v_{E_{2} H} =-3$, $v_{GH} =3$, $v_{E_{6} F} =1$ and $v_{FG_{3}} =-1$.
To avoid mistakes it is convenient to annotate on paper the $v$'s in the diagram in front of the vertices, the sign has to change along an edge.
The balanced reaction is $3E_{2}H+FG_3\to E_{6} F+ 3GH$.\footnote{The Betelgeusian civilization has been saved: all the dangerous $FG_3$ could be removed  from the atmosphere and replaced by the toxic $E_{6} F$.}
\end{example}

Instead of quivering the practitioner of ker pc can do the following: order $I$ such that 1.) for each $p\geq 1$ and $I_p<I_{p+1}$ and 2.) $I_p<I_{0}$ and order $J$ in some lexicographic order compatible with that of $J$. The order of $I$ ensures that the ker pc start trivially at $I_1$ and the complexity of the condensations grows when moving from $I_p$ to $I_{p+1}$. The order of $J$ ensures that there are no unnecessary row swaps needed.
The cost savings of this strategy and quivering seem similar.

\section{Miscellaneous investigations}
\label{sec:misc}

I would like to point out that other fundamental questions in linear algebra can be solved with pc as well. In fact, with pc one can re-purpose the calculations of less specific investigations for the more specific ones.  One can, e.g., calculate the rank of a quadratic matrix with pc. If it turns out to be full then one fills in the necessary details to calculate the determinant or the inverse. Or one figures out the rank of an arbitrary matrix with pc and later decides to determine the four subspaces, or to solve a linear system, just by filling out some extra blocks. I call such a type of analysis \emph{modular}: you can always build on what you already have achieved.

\subsection{Inhomogeneous systems}
\label{subsec:inh}

Here I am indicating how the ker pc can be adapted to solve systems of inhomogeneous linear equations
\begin{align}\label{eq:inh} 
A\vec{v} =\vec{w}
\end{align}
for $A\in R^{n\times m}$ and $\vec{w}\in R^n$.
The solution $\vec{v}$ is in $Q^m$ (actually it can be assumed to be in $(S^{-1}R)^m$ for a certain localization $S^{-1}R$ of $R$.) 
In principle, we only have to add one new row to the pattern of ker pc. One advantage is that we only have to divide at the very last step.
The basic idea is that
\begin{align*}
A\vec{v} =\vec{w} \ \ \Longleftrightarrow [ A|\vec{w}]\begin{bmatrix}
\vec{v}\\
-1
\end{bmatrix} =\vec{0},\end{align*} so that the solution of an inhomogeneous system is equivalent to an appropriately normalized vector in $\ker_Q[ A|\vec{w}]$.

In more detail, the pattern at the start of the $(l+1)$st condensation is
\begin{align*}
\begin{array}{ c|c|c }
X_{l} & Y_{l} & \vec{0}\\
\hline
(\vec{w}_{l})^{T} & (\vec{p}_{l})^{T} & \lambda _{l}
\end{array}.
\end{align*}
The initial condition is $X_{0} =A^{T}$, $Y_{0} =\boldsymbol{1}$, $\vec{w}_{0} =\vec{w}$,  $\vec{p}_{0} =\vec{0}$, $\lambda _{0} =-1$. The ker pc algorithm of Section \ref{sec:kernel} is applied to this pattern. 
The pivots however are only be chosen in the upper left block.
The recursion stops at $l=p$ when $X_{p}$ is zero or empty, which means $\Rank A=p$. If in this case  $(\vec{w}_{p})^{T} $ is nonzero, then the system has no solution. If $(\vec{w}_{p})^{T} $ is zero or empty the general solution $\vec{v}\in Q^m$ of Equation \eqref{eq:inh} is an element of
\begin{align*}
  \lambda _{p}^{-1} \vec{p}_{p}+Y_{p}^TQ^{n-p}.
\end{align*}
Note that by construction $\lambda _{p}\neq 0$ and $Y_{p}^T$ is injective.
Similarly, a solution $\vec{v}\in R^m_{\lambda_{p}}$  of Equation \eqref{eq:inh} is in $\lambda _{p}^{-1} \vec{p}_{p}+Y_{p}^TR^{n-p}_{\lambda_{p}}.$
Here $R_{\lambda_{p}}$ denotes the localization of $R$ at the multiplicative subset $S=\{\lambda_{p}^k \ | \ k\geq 1 \}$.
\begin{example}
From the pc
\begin{align*}
&\begin{array}{ c c c c|c c c c|c }
\boxed{1} & 1 & 1 & 1 & 1 &  &  &  & \\
2 & 3 & 4 & 5 &  & 1 &  &  & \\
3 & 5 & 7 & 9 &  &  & 1 &  & \\
4 & 7 & 10 & 13 &  &  &  & 1 & \\
\hline
10 & 16 & 22 & 28 &  &  &  &  & -1
\end{array} \mapsto \begin{array}{ c c c|c c c c|c }
\boxed{1}& 2 & 3 & -2 & 1 &  &  & \\
2 & 4 & 6 & -3 &  & 1 &  & \\
3 & 6 & 9 & -4 &  &  & 1 & \\
\hline
6 & 12 & 18 & -10 &  &  &  & -1
\end{array} \mapsto \begin{array}{ c c|c c c c|c }
 &  & 1 & -2 & 1 &  & \\
 &  & 2 & -3 &  & 1 & \\
\hline
 &  & 2 & -6 &  &  & -1
\end{array}\\
&=\begin{array}{ c c|c c c c|c }
 &  & 1 & -2 & 1 &  & \\
 &  & 2 & -3 &  & 1 & \\
\hline
 &  & -2 & 6 &  &  & 1
\end{array}
\end{align*}
we deduce that
$\left\{\begin{bmatrix}
v_{1}\\
v_{2}\\
v_{3}\\
v_{4}
\end{bmatrix} \in \mathbb{Q}^{4} \biggr\rvert\ \begin{bmatrix}
1 & 2 & 3 & 4\\
1 & 3 & 5 & 7\\
1 & 4 & 7 & 10\\
1 & 5 & 9 & 13
\end{bmatrix}\begin{bmatrix}
v_{1}\\
v_{2}\\
v_{3}\\
v_{4}
\end{bmatrix} =\begin{bmatrix}
10\\
16\\
22\\
28
\end{bmatrix}\right\} =\begin{bmatrix}
-2\\
6\\
0\\
0
\end{bmatrix} +\begin{bmatrix}
1&2\\
-2&-3\\
1&0\\
0&1
\end{bmatrix}\mathbb{Q}^2$. A similar statement holds over $\Z$.
\end{example}

\subsection{Matrix inversion}
\label{subsec:inversion}
Another typical problem is to find a matrix $V\in Q^{m\times k}$ that solves $AV=W$ for given matrices $A\in R^{n\times m}$ and $W\in R^{n\times k}$. Here the idea is to do the procedure of Subsection \ref{subsec:inh} simultaneously for all the columns of $W$. In other words, the pattern at the start of the $(l+1)$st condensation is
\begin{align*}
\begin{array}{ c|c|c }
X_{l} & Y_{l} & \vec{0}\\
\hline
W_l^{T} & P_{l}^{T} & \vec{\lambda}_{l}
\end{array}.
\end{align*}
The initial condition is $X_{0} =A^{T}$, $Y_{0} =\boldsymbol{1}_m$, $W_{0} =W^T$,  $P_{0} =\boldsymbol{0}$, $\vec{\lambda _{l}}=[-1,\dots,-1]^T$.
Again, the pivots can only be chosen in the upper left block.
For simplicity I assume that $A$ is injective, i.e., $n\geq m$ and $\Rank A=m$.
If $W_{m}^{T}$ is zero or empty the solution exists. In this case $V\in Q^{m\times k}$ of Equation \eqref{eq:inh} is 
  $V=P_{m}D$, where $D\in Q^{k\times k}$ is the diagonal matrix whose entries are the mutiplicative inverses of entries of $\vec{\lambda}_{m}$.  
  
In the special case when  $n=m=k$, $A$ invertible and $W=\boldsymbol 1$ the solution $P_{n}D$ exists and equals $A^{-1}$. Alternatively, if one starts the process with the invertible $X_{0} =A$ then $A^{-1}=DP_{n}^T$.

\begin{example} Let $R$ be the ring of integers $\Z$. The following pc 
\begin{align*}
&\begin{array}{ c c c c|c c c c|c }
\boxed{1} & 2 & 3 & 5 & 1 &  &  &  & \\
2 & 3 & 5 & 7 &  & 1 &  &  & \\
3 & 5 & 7 & 11 &  &  & 1 &  & \\
5 & 7 & 11 & 13 &  &  &  & 1 & \\
\hline
1 &  &  &  &  &  &  &  & -1\\
 & 1 &  &  &  &  &  &  & -1\\
 &  & 1 &  &  &  &  &  & -1\\
 &  &  & 1 &  &  &  &  & -1
\end{array} \mapsto \begin{array}{ c c c|c c c c|c }
\boxed{-1} & -1 & -3 & -2 & 1 &  &  & \\
-1 & -2 & -4 & -3 &  & 1 &  & \\
-3 & -4 & -12 & -5 &  &  & 1 & \\
\hline
-2 & -3 & -5 & -1 &  &  &  & -1\\
1 &  &  &  &  &  &  & -1\\
 & 1 &  &  &  &  &  & -1\\
 &  & 1 &  &  &  &  & -1
\end{array}\\
&\mapsto \begin{array}{ c c|c c c c|c }
\boxed{1} & 1 & 1 & 1 & -1 &  & \\
1 & 3 & -1 & 3 &  & -1 & \\
\hline
1 & -1 & -3 & 2 &  &  & 1\\
1 & 3 & 2 & -1 &  &  & 1\\
-1 &  &  &  &  &  & 1\\
 & -1 &  &  &  &  & 1
\end{array} \mapsto \begin{array}{ c|c c c c|c }
\boxed{2} & -2 & 2 & 1 & -1 & \\
\hline
-2 & -4 & 1 & 1 &  & 1\\
2 & 1 & -2 & 1 &  & 1\\
1 & 1 & 1 & -1 &  & 1\\
-1 &  &  &  &  & 1
\end{array} \mapsto \begin{array}{|c c c c|c }
\hline
-12 & 6 & 4 & -2 & 2\\
6 & -8 &  & 2 & 2\\
4 &  & -3 & 1 & 2\\
-2 & 2 & 1 & -1 & 2
\end{array}
\end{align*}
shows\footnote{To me it looks like the matrix is stretching out its arms and crawling under the fraction bar.} that 
\begin{align*}
\begin{bmatrix}
1 & 2 & 3 & 5\\
2 & 3 & 5 & 7\\
3 & 5 & 7 & 11\\
5 & 7 & 11 & 13
\end{bmatrix}^{-1} =\frac{1}{2}\begin{bmatrix}
-12 & 6 & 4 & -2\\
6 & -8 & 0 & 2\\
4 & 0 & -3 & 1\\
-2 & 2 & 1 & -1
\end{bmatrix}\in\Q^{4\times 4}.
\end{align*}
During the process there were no row swaps and the symmetry of the matrix could be maintained, which makes the procedure more effective. I suggest to call this procedure \emph{inv pc}.
\end{example}

\subsection{The four subspaces}
\label{subsec:4pc}
Given a matrix $A\in R^{n\times m}$ (or $\in Q^{n\times m}$) we can use a version of ker pc to actually determine bases for $\ker_Q(A)$, $\operatorname{im}_Q(A)$,
$\ker_Q(A^T)$ and $\operatorname{im}_Q(A^T)$ in just one calculation.

The procedure is similar to that inv pc of the previous subsection and it goes as follows. Write $A=\begin{bmatrix}
\vec{v}_{1} & \cdots  & \vec{v}_{m}
\end{bmatrix}$ \ and \ $A^{T} =\begin{bmatrix}
\vec{u}_{1} & \cdots  & \vec{u}_{m}
\end{bmatrix}$ and put $p:=\operatorname{rank}( A)$.
Before step $l+1$ of the condensation the pattern is
\begin{align*}
 \begin{array}{ c|c }
X_{l} & Y_{l}\\
\hline
W_{l}^{T} & 
\end{array}.
\end{align*}
What is in the lower right block is irrelevant for the four subspaces and we leave it blank. The starting values are $X_{0} =A^{T} ,\  Y_{0} =\boldsymbol{1}_{m} ,\ W_{0}^{T} =\boldsymbol{1}_{n}$. The ker pc algorithm of Section \ref{sec:kernel}
is applied to this pattern, while the pivots can only be chosen in the upper left block. (What appears in the lower right block is irrelvant for the recursion; the starting value is by default $\boldsymbol{0}$.)
If for $l=0$ there is a pivot in the first column add $\vec{u}_{1}$ to the list of generators of $\operatorname{im}_Q\left( A^{T}\right)$ and denote by $\sigma_{1}$ row permutation used to move the pivot to the first row. Otherwise, there is a maximal zero block of size $m\times k_{1}$ in the first columns $X_{0}$. Add the first $k_{1}$ columns in $W_{0}^{T}$ (that is the columns below the zero block) to the list of generators of $\ker_Q\left( A^{T}\right)$. There is a pivot in the $( k_{1} +1)$th column of $X_{0}$. Apply now \textbf{Step 2 (pc)} and \textbf{Step 3 (cleaning up)} and put $l=1$. The rescaling of rows  \textbf{Step 3 (cleaning up)} must not be applied below the bar. If there is a pivot in the first column add $\vec{u}_{2}$ to the list of generators of $\operatorname{im}_Q\left( A^{T}\right)$ and denote by $\sigma_{2}$ row permutation used to move the pivot to the first row. Otherwise, there is a maximal zero block of size $( m-1) \times k_{2}$ in the first columns $X_{1}$. Add the first $k_{2}$ columns in $W_{1}^{T}$ (that is the columns below the zero block) to the list of generators of $\ker_Q\left( A^{T}\right)$. There is a pivot in the $( k_{2} +1)$th column of $X_{1}$ in the first row. 
Apply now \textbf{Step 2 (pc)} and \textbf{Step 3 (cleaning up)} and put $l=2$.
In this version of pc, applying \textbf{Step 3} is not be allowed below the bar.
Continue in that manner until $l=p$. Now $X_{p}$ is zero or empty. Add the columns of $W_{p}^{T}$  to the list of generators of $\ker_Q\left( A^{T}\right)$. The transposes of the rows of $Y_{p}$ form a basis for $\ker_Q( A)$. Putting $\sigma:=\sigma_{p} \cdots \sigma_{1}$, a basis for $\operatorname{im}_Q( A)$ is formed by $\vec{v}_{\sigma^{-1}( 1)} ,\vec{v}_{\sigma^{-1}( 1)} ,\dotsc ,\vec{v}_{\sigma^{-1}( p)}$. In this way we obtain simultaneously bases for all the four $Q$-subspaces $\ker_Q( A) ,\ \operatorname{im}_Q( A) ,\ker_Q\left( A^{T}\right)$ and $\operatorname{im}_Q\left( A^{T}\right)$.

I suggest for this procedure the name \emph{4pc}.
It is actually a pretty quick and  clean way to deduce such bases. Let us look at some examples.

\begin{example}
    Let me put sort of an academic example that hopefully allows to get an impression what is going on with the row permutations:
\begin{align*}
    A=\begin{bmatrix}
0 & 0 & 0 & 1\\
0 & 0 & 1 & 1\\
0 & 0 & 1 & 1\\
1 & 0 & 1 & 1
\end{bmatrix} =\begin{bmatrix}
\vec{v}_{1} & \vec{v}_{2} & \vec{v}_{3} & \vec{v}_{4}
\end{bmatrix}, \ \ \ A^{T} =\begin{bmatrix}
0 & 0 & 0 & 1\\
0 & 0 & 0 & 0\\
0 & 1 & 1 & 1\\
1 & 1 & 1 & 1
\end{bmatrix} =\begin{bmatrix}
\vec{u}_{1} & \vec{u}_{2} & \vec{u}_{3} & \vec{u}_{4}
\end{bmatrix}.
\end{align*}

In the first step of the 4pc we apply the permutation $\sigma =( 1,3,2,4)$ to the rows of $A^{T}$. Non-mathematicians might prefer a graphical notation for $\sigma$ instead of the cycle notation used above. It makes it easy to keep track of compositions and inversions of permutations.
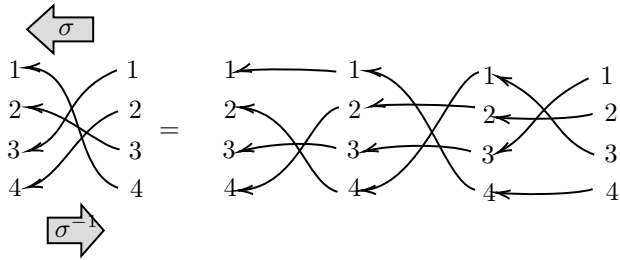
\begin{center}

\tikzset{every picture/.style={line width=0.75pt}} 

\begin{tikzpicture}[x=0.5pt,y=0.5pt,yscale=-1,xscale=1]

\draw    (67.36,82.63) .. controls (94.46,85.77) and (117.69,111.51) .. (134.39,114.32) ;
\draw [shift={(65.26,82.44)}, rotate = 4.09] [color={rgb, 255:red, 0; green, 0; blue, 0 }  ][line width=0.75]    (10.93,-3.29) .. controls (6.95,-1.4) and (3.31,-0.3) .. (0,0) .. controls (3.31,0.3) and (6.95,1.4) .. (10.93,3.29)   ;
\draw    (66.95,142.37) .. controls (91.94,133.39) and (103.04,96.16) .. (133.26,85.44) ;
\draw [shift={(65,143)}, rotate = 343.77] [color={rgb, 255:red, 0; green, 0; blue, 0 }  ][line width=0.75]    (10.93,-3.29) .. controls (6.95,-1.4) and (3.31,-0.3) .. (0,0) .. controls (3.31,0.3) and (6.95,1.4) .. (10.93,3.29)   ;
\draw    (65.3,52.45) .. controls (108.88,53.8) and (96.7,130.75) .. (133.26,143.44) ;
\draw [shift={(63.26,52.44)}, rotate = 358.75] [color={rgb, 255:red, 0; green, 0; blue, 0 }  ][line width=0.75]    (10.93,-3.29) .. controls (6.95,-1.4) and (3.31,-0.3) .. (0,0) .. controls (3.31,0.3) and (6.95,1.4) .. (10.93,3.29)   ;
\draw    (67.22,114.83) .. controls (98.3,104.43) and (90.25,71.96) .. (132.39,54.32) ;
\draw [shift={(65.26,115.44)}, rotate = 343.83] [color={rgb, 255:red, 0; green, 0; blue, 0 }  ][line width=0.75]    (10.93,-3.29) .. controls (6.95,-1.4) and (3.31,-0.3) .. (0,0) .. controls (3.31,0.3) and (6.95,1.4) .. (10.93,3.29)   ;
\draw  [fill={rgb, 255:red, 0; green, 0; blue, 0 }  ,fill opacity=0.14 ] (80,165.99) -- (105.43,165.99) -- (105.43,156.32) -- (122.39,175.66) -- (105.43,195) -- (105.43,185.33) -- (80,185.33) -- cycle ;
\draw  [fill={rgb, 255:red, 0; green, 0; blue, 0 }  ,fill opacity=0.14 ] (114.39,14.39) -- (84.39,14.39) -- (84.39,5.32) -- (64.39,23.47) -- (84.39,41.61) -- (84.39,32.54) -- (114.39,32.54) -- cycle ;
\draw    (422.01,59.92) .. controls (465.25,80.1) and (460.88,105.51) .. (493.39,118.32) ;
\draw [shift={(420,59)}, rotate = 24.12] [color={rgb, 255:red, 0; green, 0; blue, 0 }  ][line width=0.75]    (10.93,-3.29) .. controls (6.95,-1.4) and (3.31,-0.3) .. (0,0) .. controls (3.31,0.3) and (6.95,1.4) .. (10.93,3.29)   ;
\draw    (421.78,116.06) .. controls (450.12,100.64) and (447.12,81.02) .. (489.26,60.44) ;
\draw [shift={(420,117)}, rotate = 332.79] [color={rgb, 255:red, 0; green, 0; blue, 0 }  ][line width=0.75]    (10.93,-3.29) .. controls (6.95,-1.4) and (3.31,-0.3) .. (0,0) .. controls (3.31,0.3) and (6.95,1.4) .. (10.93,3.29)   ;
\draw    (421.01,90.23) .. controls (434.24,91.53) and (476.2,92.16) .. (492.26,87.44) ;
\draw [shift={(419,90)}, rotate = 7.97] [color={rgb, 255:red, 0; green, 0; blue, 0 }  ][line width=0.75]    (10.93,-3.29) .. controls (6.95,-1.4) and (3.31,-0.3) .. (0,0) .. controls (3.31,0.3) and (6.95,1.4) .. (10.93,3.29)   ;
\draw    (421.01,147.23) .. controls (434.24,148.53) and (476.2,149.16) .. (492.26,144.44) ;
\draw [shift={(419,147)}, rotate = 7.97] [color={rgb, 255:red, 0; green, 0; blue, 0 }  ][line width=0.75]    (10.93,-3.29) .. controls (6.95,-1.4) and (3.31,-0.3) .. (0,0) .. controls (3.31,0.3) and (6.95,1.4) .. (10.93,3.29)   ;
\draw    (227.86,82.74) .. controls (269.47,88.54) and (274.64,139.69) .. (299.26,146.44) ;
\draw [shift={(225.26,82.44)}, rotate = 5.19] [color={rgb, 255:red, 0; green, 0; blue, 0 }  ][line width=0.75]    (10.93,-3.29) .. controls (6.95,-1.4) and (3.31,-0.3) .. (0,0) .. controls (3.31,0.3) and (6.95,1.4) .. (10.93,3.29)   ;
\draw    (226.27,145.31) .. controls (269.54,141.77) and (278.59,87.36) .. (300.26,82.44) ;
\draw [shift={(224.26,145.44)}, rotate = 357.46] [color={rgb, 255:red, 0; green, 0; blue, 0 }  ][line width=0.75]    (10.93,-3.29) .. controls (6.95,-1.4) and (3.31,-0.3) .. (0,0) .. controls (3.31,0.3) and (6.95,1.4) .. (10.93,3.29)   ;
\draw    (225.34,55.01) .. controls (238.26,52.95) and (282.28,53.56) .. (298.26,55.44) ;
\draw [shift={(223.26,55.44)}, rotate = 344.54] [color={rgb, 255:red, 0; green, 0; blue, 0 }  ][line width=0.75]    (10.93,-3.29) .. controls (6.95,-1.4) and (3.31,-0.3) .. (0,0) .. controls (3.31,0.3) and (6.95,1.4) .. (10.93,3.29)   ;
\draw    (226.99,115.43) .. controls (245.83,110.24) and (280.89,107.65) .. (298.26,113.44) ;
\draw [shift={(225,116)}, rotate = 343.06] [color={rgb, 255:red, 0; green, 0; blue, 0 }  ][line width=0.75]    (10.93,-3.29) .. controls (6.95,-1.4) and (3.31,-0.3) .. (0,0) .. controls (3.31,0.3) and (6.95,1.4) .. (10.93,3.29)   ;
\draw    (323.87,55.24) .. controls (365.78,62.04) and (378.64,138.16) .. (403.26,144.91) ;
\draw [shift={(321.26,54.91)}, rotate = 5.19] [color={rgb, 255:red, 0; green, 0; blue, 0 }  ][line width=0.75]    (10.93,-3.29) .. controls (6.95,-1.4) and (3.31,-0.3) .. (0,0) .. controls (3.31,0.3) and (6.95,1.4) .. (10.93,3.29)   ;
\draw    (319.27,145.29) .. controls (362.92,140.98) and (384.59,60.83) .. (406.26,55.91) ;
\draw [shift={(317.26,145.44)}, rotate = 357.46] [color={rgb, 255:red, 0; green, 0; blue, 0 }  ][line width=0.75]    (10.93,-3.29) .. controls (6.95,-1.4) and (3.31,-0.3) .. (0,0) .. controls (3.31,0.3) and (6.95,1.4) .. (10.93,3.29)   ;
\draw    (326.38,81.49) .. controls (339.73,79.48) and (387.28,80.56) .. (403.26,82.44) ;
\draw [shift={(324.26,81.91)}, rotate = 344.54] [color={rgb, 255:red, 0; green, 0; blue, 0 }  ][line width=0.75]    (10.93,-3.29) .. controls (6.95,-1.4) and (3.31,-0.3) .. (0,0) .. controls (3.31,0.3) and (6.95,1.4) .. (10.93,3.29)   ;
\draw    (321.01,115.43) .. controls (340.38,110.41) and (382.89,110.12) .. (400.26,115.91) ;
\draw [shift={(319,116)}, rotate = 343.06] [color={rgb, 255:red, 0; green, 0; blue, 0 }  ][line width=0.75]    (10.93,-3.29) .. controls (6.95,-1.4) and (3.31,-0.3) .. (0,0) .. controls (3.31,0.3) and (6.95,1.4) .. (10.93,3.29)   ;

\draw (137,45.4) node [anchor=north west][inner sep=0.75pt]    {$1$};
\draw (139,75.4) node [anchor=north west][inner sep=0.75pt]    {$2$};
\draw (48,75.4) node [anchor=north west][inner sep=0.75pt]    {$2$};
\draw (48,45.4) node [anchor=north west][inner sep=0.75pt]    {$1$};
\draw (139,106.4) node [anchor=north west][inner sep=0.75pt]    {$3$};
\draw (47,105.4) node [anchor=north west][inner sep=0.75pt]    {$3$};
\draw (140,135.4) node [anchor=north west][inner sep=0.75pt]    {$4$};
\draw (48,134.4) node [anchor=north west][inner sep=0.75pt]    {$4$};
\draw (83,162) node [anchor=north west][inner sep=0.75pt]    {$\sigma ^{-1}$};
\draw (86,17) node [anchor=north west][inner sep=0.75pt]    {$\sigma $};
\draw (496,49.4) node [anchor=north west][inner sep=0.75pt]    {$1$};
\draw (499,77.4) node [anchor=north west][inner sep=0.75pt]    {$2$};
\draw (407,79.4) node [anchor=north west][inner sep=0.75pt]    {$2$};
\draw (407,49.4) node [anchor=north west][inner sep=0.75pt]    {$1$};
\draw (499,108.4) node [anchor=north west][inner sep=0.75pt]    {$3$};
\draw (406,109.4) node [anchor=north west][inner sep=0.75pt]    {$3$};
\draw (500,137.4) node [anchor=north west][inner sep=0.75pt]    {$4$};
\draw (407,138.4) node [anchor=north west][inner sep=0.75pt]    {$4$};
\draw (211,75.4) node [anchor=north west][inner sep=0.75pt]    {$2$};
\draw (211,45.4) node [anchor=north west][inner sep=0.75pt]    {$1$};
\draw (210,105.4) node [anchor=north west][inner sep=0.75pt]    {$3$};
\draw (211,134.4) node [anchor=north west][inner sep=0.75pt]    {$4$};
\draw (305,75.4) node [anchor=north west][inner sep=0.75pt]    {$2$};
\draw (305,45.4) node [anchor=north west][inner sep=0.75pt]    {$1$};
\draw (304,105.4) node [anchor=north west][inner sep=0.75pt]    {$3$};
\draw (305,134.4) node [anchor=north west][inner sep=0.75pt]    {$4$};
\draw (162,95.4) node [anchor=north west][inner sep=0.75pt]    {$=$};
\end{tikzpicture}

\end{center}
Written in cycle form this is $ \sigma =( 1,3,2,4) =( 2,4)( 1,4)( 1,3)$ and $\sigma^{-1} =( 4,2,3,1)$.
The 4pc calculation goes as follows
\begin{align*}
&\begin{array}{ c c c c|c c c c }
 &  &  & 1 & 1 &  &  & \\
 &  &  &  &  & 1 &  & \\
 & 1 & 1 & 1 &  &  & 1 & \\
1 & 1 & 1 & 1 &  &  &  & 1\\
\hline
1 &  &  &  &  &  &  & \\
 & 1 &  &  &  &  &  & \\
 &  & 1 &  &  &  &  & \\
 &  &  & 1 &  &  &  & 
\end{array} \ \ \ \ \ \ \ \ \ \ \ \ \ \ \ \left(\begin{bmatrix}
0\\
0\\
0\\
1
\end{bmatrix}\rightarrow \operatorname{im}_\Q\left( A^{T}\right)\right)\\
&\overset{\sigma }{\mapsto }\begin{array}{ c c c c|c c c c }
\boxed{1} & 1 & 1 & 1 &  &  &  & 1\\
 & 1 & 1 & 1 &  &  & 1 & \\
 &  &  & 1 & 1 &  &  & \\
 &  &  &  &  & 1 &  & \\
\hline
1 &  &  &  &  &  &  & \\
 & 1 &  &  &  &  &  & \\
 &  & 1 &  &  &  &  & \\
 &  &  & 1 &  &  &  & 
\end{array} \ \ \ \ \ \ \ \mapsto \begin{array}{ c c c|c c c c }
\boxed{1} & 1 & 1 &  &  & 1 & \\
 &  & 1 & 1 &  &  & \\
 &  &  &  & 1 &  & \\
\hline
-1 & -1 & -1 &  &  &  & \\
1 &  &  &  &  &  & \\
 & 1 &  &  &  &  & \\
 &  & 1 &  &  &  & 
\end{array} \ \ \ \ \ \ \ \left(\begin{bmatrix}
0\\
0\\
1\\
1
\end{bmatrix}\rightarrow \operatorname{im}_\Q\left( A^{T}\right)\right)\\
&\mapsto\begin{array}{ c c|c c c c }
 & \boxed{1} & 1 &  &  & \\
 &  &  & 1 &  & \\
\hline
 &  &  &  &  & \\
-1 & -1 &  &  &  & \\
1 &  &  &  &  & \\
 & 1 &  &  &  & 
\end{array} \ \ \ \ \ \ \ \ \ \ \ \left(\begin{bmatrix}
0\\
-1\\
1\\
0
\end{bmatrix}\rightarrow \ker_\Q\left( A^{T}\right) ,\ \begin{bmatrix}
1\\
0\\
1\\
1
\end{bmatrix}\rightarrow \operatorname{im}_\Q\left( A^{T}\right)\right) \ \ \ \ \ \ \\
&\mapsto \begin{array}{|c c c c }
 & 1 &  & \\
\hline
 &  &  & \\
 &  &  & \\
 &  &  & \\
 &  &  & 
\end{array},
\end{align*}
and hence
\begin{align*}
&\ker_\Q( A)=\operatorname{\mathbb{Q}}\begin{bmatrix}
0\\
1\\
0\\
0
\end{bmatrix} ,\ \ker_\Q\left( A^{T}\right) =\operatorname{\mathbb{Q}}\begin{bmatrix}
0\\
-1\\
1\\
0
\end{bmatrix} ,\ \ \operatorname{im}_\Q( A) =\operatorname{im}_\Q\left(\begin{bmatrix}
\vec{v}_{\sigma ^{-1}( 1)} & \vec{v}_{\sigma ^{-1}( 2)} & \vec{v}_{\sigma ^{-1}( 3)}
\end{bmatrix}\right) =\begin{bmatrix}
1 & 0 & 0\\
1 & 1 & 0\\
1 & 1 & 0\\
1 & 1 & 1
\end{bmatrix}\mathbb{Q}^{3} ,\\
&\operatorname{im}_\Q\left( A^{T}\right) =\operatorname{im}_\Q\left(\begin{bmatrix}
\vec{u}_{1} & \vec{u}_{2} & \vec{u}_{4}
\end{bmatrix}\right) =\begin{bmatrix}
0 & 0 & 1\\
0 & 0 & 0\\
0 & 1 & 1\\
1 & 1 & 1
\end{bmatrix}\mathbb{Q}^{3} .
\end{align*}
\end{example}

\begin{example}
    Let us consider\footnote{The matrix $ A$ is a piece of the Hankel matrix of the tribonacci numbers (entry A000073 in \cite{OEIS}).}
\begin{align*}
A=\begin{bmatrix}
0 & 1 & 1 & 2\\
1 & 1 & 2 & 4\\
1 & 2 & 4 & 7\\
2 & 4 & 7 & 13\\
4 & 7 & 13 & 24
\end{bmatrix} =\begin{bmatrix}
\vec{v}_{1} & \vec{v}_{2} & \vec{v}_{3} & \vec{v}_{4}
\end{bmatrix} ,\ A^{T} =\begin{bmatrix}
0 & 1 & 1 & 2 & 4\\
1 & 1 & 2 & 4 & 7\\
1 & 2 & 4 & 7 & 13\\
2 & 4 & 7 & 13 & 24
\end{bmatrix} =\begin{bmatrix}
\vec{u}_{1} & \vec{u}_{2} & \vec{u}_{3} & \vec{u}_{4} & \vec{u}_{5}
\end{bmatrix}.
\end{align*}
At the first step we apply the transposition $ \sigma =(1,2)$ (i.e., a switch of rows):
\begin{center}

\tikzset{every picture/.style={line width=0.75pt}} 

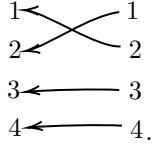
\begin{tikzpicture}[x=0.5pt,y=0.5pt,yscale=-1,xscale=1]

\draw    (31.18,22.23) .. controls (47.55,24.96) and (83.94,50.46) .. (102.08,49.5) ;
\draw [shift={(29,22)}, rotate = 1.91] [color={rgb, 255:red, 0; green, 0; blue, 0 }  ][line width=0.75]    (10.93,-3.29) .. controls (6.95,-1.4) and (3.31,-0.3) .. (0,0) .. controls (3.31,0.3) and (6.95,1.4) .. (10.93,3.29)   ;
\draw    (32.03,83.78) .. controls (46.22,82.41) and (84.93,81.55) .. (101.08,82.5) ;
\draw [shift={(30,84)}, rotate = 352.94] [color={rgb, 255:red, 0; green, 0; blue, 0 }  ][line width=0.75]    (10.93,-3.29) .. controls (6.95,-1.4) and (3.31,-0.3) .. (0,0) .. controls (3.31,0.3) and (6.95,1.4) .. (10.93,3.29)   ;
\draw    (32.16,52.54) .. controls (48.13,48.41) and (78.34,26.39) .. (101.08,23.5) ;
\draw [shift={(30,53)}, rotate = 350.79] [color={rgb, 255:red, 0; green, 0; blue, 0 }  ][line width=0.75]    (10.93,-3.29) .. controls (6.95,-1.4) and (3.31,-0.3) .. (0,0) .. controls (3.31,0.3) and (6.95,1.4) .. (10.93,3.29)   ;
\draw    (34.03,110.78) .. controls (48.22,109.41) and (86.93,108.55) .. (103.08,109.5) ;
\draw [shift={(32,111)}, rotate = 352.94] [color={rgb, 255:red, 0; green, 0; blue, 0 }  ][line width=0.75]    (10.93,-3.29) .. controls (6.95,-1.4) and (3.31,-0.3) .. (0,0) .. controls (3.31,0.3) and (6.95,1.4) .. (10.93,3.29)   ;

\draw (104,13.4) node [anchor=north west][inner sep=0.75pt]    {$1$};
\draw (106,43.4) node [anchor=north west][inner sep=0.75pt]    {$2$};
\draw (15,43.4) node [anchor=north west][inner sep=0.75pt]    {$2$};
\draw (15,13.4) node [anchor=north west][inner sep=0.75pt]    {$1$};
\draw (106,74.4) node [anchor=north west][inner sep=0.75pt]    {$3$};
\draw (14,73.4) node [anchor=north west][inner sep=0.75pt]    {$3$};
\draw (107,103.4) node [anchor=north west][inner sep=0.75pt]    {$4$};
\draw (15,102.4) node [anchor=north west][inner sep=0.75pt]    {$4$};
\end{tikzpicture}.
\end{center}
Here I am already trying to write a bit less, when recording the steps of the 4pc:
\begin{align*}
&\begin{array}{ c c c c c|c c c c }
 & 1 & 1 & 2 & 4 & 1 &  &  & \\
1 & 1 & 2 & 4 & 7 &  & 1 &  & \\
1 & 2 & 4 & 7 & 13 &  &  & 1 & \\
2 & 4 & 7 & 13 & 24 &  &  &  & 1\\
\hline
1 &  &  &  &  &  &  &  & \\
 & 1 &  &  &  &  &  &  & \\
 &  & 1 &  &  &  &  &  & \\
 &  &  & 1 &  &  &  &  & \\
 &  &  &  & 1 &  &  &  & 
\end{array}\overset{\sigma }{\mapsto }\begin{array}{ c c c c c|c c c c }
\boxed{1} & 1 & 2 & 4 & 7 &  & 1 &  & \\
 & 1 & 1 & 2 & 4 & 1 &  &  & \\
1 & 2 & 4 & 7 & 13 &  &  & 1 & \\
2 & 4 & 7 & 13 & 24 &  &  &  & 1\\
\hline
1 &  &  &  &  &  &  &  & \\
 & 1 &  &  &  &  &  &  & \\
 &  & 1 &  &  &  &  &  & \\
 &  &  & 1 &  &  &  &  & \\
 &  &  &  & 1 &  &  &  & 
\end{array} \mapsto \begin{array}{ c c c c|c c c c }
\boxed{1} & 1 & 2 & 4 & 1 &  &  & \\
1 & 2 & 3 & 6 &  & -1 & 1 & \\
2 & 3 & 5 & 10 &  & -2 &  & 1\\
\hline
-1 & -2 & -4 & -7 &  &  &  & \\
1 &  &  &  &  &  &  & \\
 & 1 &  &  &  &  &  & \\
 &  & 1 &  &  &  &  & \\
 &  &  & 1 &  &  &  & 
\end{array}\\
&\mapsto \begin{array}{ c c c|c c c c }
\boxed{1} & 1 & 2 & -1 & -1 & 1 & \\
1 & 1 & 2 & -2 & -2 &  & 1\\
\hline
-1 & -2 & -3 &  &  &  & \\
-1 & -2 & -4 &  &  &  & \\
1 &  &  &  &  &  & \\
 & 1 &  &  &  &  & \\
 &  & 1 &  &  &  & 
\end{array} \mapsto \begin{array}{ c c|c c c c }
 &  & -1 & -1 & -1 & 1\\
\hline
-1 & -1 &  &  &  & \\
-1 & -2 &  &  &  & \\
-1 & -2 &  &  &  & \\
1 & 0 &  &  &  & \\
0 & 1 &  &  &  & 
\end{array}.
\end{align*}
We conclude that
\begin{align*}
&\ker_\Q( A)=\operatorname{\mathbb{Q}}\begin{bmatrix}
-1\\
-1\\
-1\\
1
\end{bmatrix} ,\ \ker_\Q\left( A^{T}\right) =\begin{bmatrix}
-1 & -1\\
-1 & -2\\
-1 & -2\\
1 & 0\\
0 & 1
\end{bmatrix} \mathbb{Q}^{2} ,\\
&\operatorname{im}_\Q( A) =\operatorname{im}_\Q\left(\begin{bmatrix}
\vec{v}_{\sigma ^{-1}( 1)} & \vec{v}_{\sigma ^{-1}( 2)} & \vec{v}_{\sigma ^{-1}( 3)}
\end{bmatrix}\right) =\begin{bmatrix}
1 & 0 & 1\\
1 & 1 & 2\\
2 & 1 & 4\\
4 & 2 & 7\\
7 & 4 & 13
\end{bmatrix}\mathbb{Q}^{3} ,\\
&\operatorname{im}_\Q\left( A^{T}\right) =\operatorname{im}_\Q\left(\begin{bmatrix}
\vec{u}_{1} & \vec{u}_{2} & \vec{u}_{4}
\end{bmatrix}\right) =\begin{bmatrix}
0 & 1 & 1\\
1 & 1 & 2\\
1 & 2 & 4\\
2 & 4 & 7
\end{bmatrix}\mathbb{Q}^{3} .\\
\end{align*}
\end{example}

Combining the ideas of Subsections \ref{subsec:inh}, \ref{subsec:inversion} and \ref{subsec:4pc} one can use pc to calculate the general solution to the problem of finding the one-sided inverse of a full rank matrix. The reader is invited to work out the details.

\section{Proof of ker pc and 4pc}
\label{sec:proofs}
Let $R$ be a gcd domain (e.g., a PID).
The following argument is based on an idea that I learned from stackexchange user Ben Grossman \cite{Grossman}. 

It is enough to prove 4pc as ker pc is part of it. Recall that one says a matrix $A=[ a_{ij}]$ is in \emph{echelon form} if 1) for pivots $a_{ij} ,\ a_{i'j'}$ with $i< i'$ it follows that $j< j'$ and 2) there are no zero rows above a pivot. (A pivot is a non-zero entry $a_{ij}$ such that  $a_{ij'} =0$ for all $j'< j$.) I will define an iteration that brings $A^{T}$ into echelon form. This will be done by the following operations: permuting rows of $A^{T}$, adding $R$-multiples of rows to other rows and dividing rows by their gcd's.

By the permutation matrix of the permutation $\sigma\in\operatorname{Aut}(\{1,\dots,m\})$  we mean $P_{\sigma} :=\sum_{i=1}^m E_{\sigma(i)i}\in R^{m\times m}$.
By a \emph{pc-op matrix} we mean a matrix of the form 
\begin{align*}
M_{i,j}( a,b) =(\boldsymbol{1} -bE_{ji } +( a-1) E_{jj}) \in R^{m\times m}
\end{align*}
for some $a,b\in R$, $i,j\in \{1,\dotsc ,m\}$. Here $E_{ij}$ is the $m\times m$ matrix whose only nonzero entry is a $1$ in line $i$ and column $j$. For $ i,j,k$ pairwise distinct and $ a,b,a',b'\in R$ we easily verify $M_{i,j}( a,b) M_{i,k}( a',b') =M_{i,k}( a',b') M_{i,j}( a,b)$.



For $U \in R^{m\times n}$ and $ 1\leq j\leq n,\ 1\leq \lambda <r\leq m$ we define for the integer $ k\geq m$
\begin{align*}
\Xi _{\lambda ,r,j}^{( k)}(U) :=M_{\lambda ,r}( u_{\lambda j} ,u_{rj}) =\boldsymbol{1} -u_{rj} E_{r\lambda } +( u_{\lambda j} -1) E_{rr} \in R^{k\times k}.\end{align*} 
Let $ B\in R^{n\times m}$, choose a permutation $ \sigma \in \operatorname{Aut}(\{1,\dotsc ,m\})$ and put
\begin{align*} L_{\lambda ,B^{T} ,\sigma } :=\prod _{r=\lambda +1}^{m} \Xi _{\lambda ,r,j}^{( m)}\left( P_{\sigma } B^{T}\right) \in R^{m\times m} ,\\
R_{\lambda ,B^{T} ,\sigma } :=\prod _{r=\lambda +1}^{m} \Xi _{\lambda ,r,j}^{( n)}\left( B^{T} P_{\sigma }\right) \in R^{n\times n}.
\end{align*} 
For the main argument we will use the following.

\begin{lemma} \label{lem:main}
With the notations above we have

    \begin{enumerate}
        \item $ L_{\lambda ,B^{T} ,\sigma } P_{\sigma } B^{T} -P_{\sigma } B^{T}( R_{\lambda ,B,\sigma ^{-1}})^{T} =\left[\begin{array}{ c|c }
* & *\\
\hline
* & \boldsymbol{0}
\end{array}\right]$, where $ \boldsymbol{0} \in R^{( m-\lambda ) \times ( n-\lambda )}$,
\item $ \left( L_{\lambda ,B^{T} ,\sigma } P_{\sigma } B^{T}\right)_{rj} \ =0$ for $ r >\lambda $ and
\item $ \left( P_{\sigma } B^{T} R_{\lambda ,B^{T} ,\sigma }\right)_{jr} =0$ \ for $ r >\lambda $.
    \end{enumerate}
\end{lemma}

\begin{proof}
    We evaluate
\begin{align*}
L_{\lambda ,B^{T} ,\sigma } &=\prod _{r=\lambda +1}^{m} \Xi _{\lambda ,r,j}^{( m)}\left( P_{\sigma } B^{T}\right) =\prod _{r=\lambda +1}^{m} \Xi _{\lambda ,r,j}^{( m)}\left(\sum _{p,q} b_{\sigma ( q) p} E_{pq}\right)\\
&=\prod _{r=\lambda +1}^{m} M_{\lambda ,r}( u_{\lambda j} ,u_{rj}) =\prod _{r=\lambda +1}^{m}(\boldsymbol{1} +( b_{\sigma ( j) \lambda } -1) E_{rr} -b_{\sigma ( j) r} E_{r\lambda })
\end{align*}
Here $ U=P_{\sigma } B^{T}$ and $ u_{\lambda j} =b_{\sigma ( j) \lambda } ,\ u_{rj} =b_{\sigma ( j) r}$. Expanding the product we obtain
\begin{align*}
L_{\lambda ,B^{T} ,\sigma } =\boldsymbol{1} +\sum _{r=\lambda +1}^{m}( ( b_{\sigma ( j) \lambda } -1) E_{rr} -b_{\sigma ( j) r} E_{r\lambda }) =\sum _{r=1}^{\lambda } E_{rr} +\sum _{r=\lambda +1}^{m}( b_{\sigma ( j) \lambda } E_{rr} -b_{\sigma ( j) r} E_{r\lambda }) .
\end{align*}
Hence
\begin{align}
\label{eq:L}
L_{\lambda ,B^{T} ,\sigma } P_{\sigma } B^{T} &=\sum _{p,q}\left(\sum _{r=1}^{\lambda } E_{rr} +\sum _{r=\lambda +1}^{m}( \ b_{\sigma ( j) \lambda } E_{rr} -b_{\sigma ( j) r} E_{r\lambda })\right) b_{\sigma ( q) p} E_{pq}\\ 
\nonumber&=\sum _{q}\sum _{r=1}^{\lambda } b_{\sigma ( q) r} E_{rq} +\sum _{q}\sum _{r=\lambda +1}^{m}( b_{\sigma ( j) \lambda } b_{\sigma ( q) r} -b_{\sigma ( j) r} b_{\sigma ( q) \lambda }) E_{rq} .
\end{align}
If $ q=j$ the second sum is zero proving 2).
Next, we look at
\begin{align*}
 R_{\lambda ,B^{T} ,\sigma } &:=\prod _{r=\lambda +1}^{m} \Xi _{\lambda ,r,j}^{( n)}\left( B^{T} P_{\sigma }\right) =\prod _{r=\lambda +1}^{m} \Xi _{\lambda ,r,j}^{( n)}\left(\sum _{p,q} b_{\sigma ( q) p} E_{pq}\right)\\
 &=\prod _{r=\lambda +1}^{m} M_{\lambda ,r}( u_{\lambda j} ,u_{rj}) =\prod _{r=\lambda +1}^{m}(\boldsymbol{1} +( b_{\sigma ( j) \lambda } -1) E_{rr} -b_{\sigma ( j) r} E_{r\lambda }).
 \end{align*}
Here $ U=B^{T} P_{\sigma }$ and $ u_{\lambda j} =b_{\sigma ( j) \lambda } ,\ u_{rj} =b_{r\sigma ^{-1}( j)} .$ Expanding the product we obtain
\begin{align*}
 R_{\lambda ,B^{T} ,\sigma } &=\boldsymbol{1} +\sum _{r=\lambda +1}^{m}(( b_{\sigma ( j) \lambda } -1) E_{rr} -b_{r\sigma ^{-1}( j)} E_{r\lambda }) =\sum _{r=1}^{\lambda } E_{rr} +\sum _{r=\lambda +1}^{m}( b_{\sigma ( j) \lambda } E_{rr} -b_{\sigma ( j) r} E_{r\lambda })
 \end{align*}
 and hence obtain
\begin{align}
\label{eq:R}
 P_{\sigma } B^{T}( R_{\lambda ,B^{T} ,\sigma })^{T} &=\sum _{p,q} b_{\sigma ( q) p} E_{pq}\left(\sum _{r=1}^{\lambda } E_{rr} +\sum _{r=\lambda +1}^{m}( b_{\sigma ( j) \lambda } E_{rr} -b_{\sigma ( j) r} E_{\lambda r})\right)\\
 \nonumber
&=\sum _{p} b_{\sigma ( r) p} E_{pr} +\sum _{p}\sum _{r=\lambda +1}^{m}( b_{\sigma ( r) p} b_{\sigma ( j) \lambda } -b_{\sigma ( \lambda ) p} b_{\sigma ( j) r}) E_{pr} .
\end{align}
If $ p=j$ the second sum is zero proving 3). For $ p,q >\lambda $ in both formulas, Eqns. \eqref{eq:L} and \eqref{eq:R}, the second 
second terms coincide, which proves 1).
\end{proof}

Note that the $2\times 2$-minors in Eqns. \eqref{eq:L} and \eqref{eq:R} are exactly those of pivotal condensation (see Step 1 and 2). So the scheme of pc can be realized by left multiplication with a product of pc-op matrices, or, alternatively,  by right multiplication with a product of transposes of pc-op matrices. The catch is that the other terms in Eqns. \eqref{eq:L} and \eqref{eq:R} do not coincide. This does not cause any serious problems, however.

Now for $ 0\leq l\leq p-1$ we set $ L_{l+1} :=L_{l+1,A_{l}^{T} ,\sigma _{l+1}} ,\ N_{l+1} :=R_{l+1,A_{l}^{T} ,\sigma _{l+1}} $  where $A_{l}^{T}$ is defined by the recursion $A_{l+1}^{T}= L_{l} P_{\sigma _{l}}A_{l}^{T}$, $A_0^T=A^T$. The permutations $\sigma _{l+1}$ are the ones from Subsection \ref{subsec:4pc}.
We write $ M_{l+1} =L_{l+1} P_{\sigma _{l+1}}$. By Lemma \ref{lem:main} 2) the matrix
\begin{align*}
E=M_{p} \cdots M_{2}M_{1} A^{T}\in R{^{m\times n}}
\end{align*}
is in echelon form and of rank $ p$. So the last $ n-p$ rows of $ E$ are zero. The lower $ ( n-p) \times \left( n-p -\sum _{i=1}^{p} k_{i}\right)$-block of $ E$ is just $ X_{p}$, which is zero or empty. The lower $( m-p) \times n$-block of $ M_{p} \cdots M_{2} M_{1}$ multiplied from the left by an appropriate diagonal matrix equals $Y_{p}$. But since the last $m-p$ columns of $E^{T}$ are zero we must have $ AY_{p}^{T} =\boldsymbol{0}$. Note that $ Y_{p}^{T}$ is by construction of full rank, so that $ \ker_Q( A) =Y_{p}^{T} Q^{m-p}$.

On the other hand, the matrix
\begin{align*}
F=[ f_{ij}] :=P_{\sigma _{p}} \cdots P_{\sigma _{2}} P_{\sigma _{1}} A^{T} N_{1} N_{2}\cdots N_{p} \in R{^{m\times n}}
\end{align*}
has by Lemma \ref{lem:main} 3) the property (P): for transposed pivots $ f_{ij} ,\ f_{kl}$ with $ l >j$ it follows $ k >i$. 
Here we say that an entry in $F$ is a \emph{transposed pivot} if it is a pivot in $F^T$.
Let $I_{F} =:\{i_{1} ,\dotsc ,i_{n-p}\}$ be the set of all column indices of zero columns of $F$. We deduce from property (P) that $ \ker( F) =R\{\vec{e}_{i} \mid i\in I_{F}\}$. By Lemma \ref{lem:main} 1) and 3) the set $I_{F}$ consists of exactly those column indices where one encounters zero columns in some $ X_{l}$. Here it is understood that the columns of $ X_{l}$ are counted starting with column index $\sum _{i=1}^{l-1}( k_{i} +1)$. With $ N=\begin{bmatrix}
\vec{v}_{1} & \dotsc  & \vec{v}_{n}
\end{bmatrix} :=N_{1}N_{2} \cdots N_{p}$ put $ N':=\begin{bmatrix}
\vec{v}_{i_{1}} & \dotsc  & \vec{v}_{i_{n-p}}
\end{bmatrix}$. It follows that $ A^{T} N'=\boldsymbol{0}$, where $ N'$ is of full rank by construction. Hence $ \ker_Q\left( A^{T}\right) =N'Q^{n-p}$.

The first $p$ columns of $E^T=AM_{1}^{T} \cdots M_{p}^{T}$ are linearly independent. Multiplying from the right by transposed pc-op or diagonal matrices does not mess up linear independence of columns. Now from 
\begin{align*}
A=E^T(M_{p}^{T})^{-1} \cdots (M_{1}^{T})^{-1}\end{align*} we see that the column vectors of $A$ with column indices $\sigma ^{-1}( 1) ,\dotsc ,\sigma ^{-1}( p)$ form a  maximally independent set of vectors in $\operatorname{im}_Q(A)$. Here it is important to keep in mind that a permutation $\sigma$ acts on columns by right multiplying with $P_{\sigma^{-1}}$.
The statement about the image of $A^T$ is, in principle, the usual one from Gaussian elimination. In fact,  with $A^{T}=\begin{bmatrix}
\vec{u}_{1} & \vec{u}_{2} & \cdots & \vec{u}_{n}
\end{bmatrix}$ we have $\operatorname{im}_Q(A^T)=R\{\vec{u}_{i}\mid \mbox{there is a transposed pivot in the $i$th column of $F$}\}$.
This finishes the proof of 4pc.



In general, it is not guaranteed that $\ker_R\left( A^{T}\right)=Y_{p}^{T}R^n$. A necessary and sufficient criterion for the case when $R$ is a PID is provided by Theorem \ref{thm:characterization}. If $R$ is a domain but not a gcd domain it is not guaranteed that the normalizations (i.e., the cleaning up of Step 3) of ker pc are theoretically feasible, even though practically they might be.
If in this situation $Y_{p}^{T}$ contains a $k\times k$-submatrix $DP$, where $D$ is a diagonal matrix of units and $P$ a permutation matrix then $\ker_R\left( A^{T}\right)=Y_{p}^{T}R^n$ follows. The reason are explained in the next section.

\section{The Kernel as a free $R$-module}
\label{sec:Rbases}

For those who would like understand how to calculate a basis for the kernel $\ker_\Z(A)$ of and integer matrix $A\in \mathbb Z^{n\times m}$  I collect here some material from commutative algebra. I give a practical account of Smith normal form and  discuss the saturation of a submodule of a finitely generated free module. 
Without some experience with Smith normal form the reader 
will not be able to appreciate the main point: Subsection \ref{subsec:Rbasis}, where I indicate how to turn a $\Q$-basis obtained by ker pc into a $\Z$-basis.
A clear and accessible reference for the background material is \cite[Section 3.7]{JacobsonBA1}, \cite[Sections 7.2,7.3]{JacobsonBA2}. Another good reference on Smith normal forms is \cite{IntegralMatrices}.

\subsection{A practical guide to Smith normal form}
\label{subsec:Smith}

In Gaussian elimination over a field one uses the following operations
\begin{enumerate}[label=\arabic*.)]
    \item switching rows, 
\item adding a multiple of a row to another row, 
\item and multiplying a row with a nonzero element.
\end{enumerate}

Working over a PID $R$, it is a natural idea to use operation 3.) only for units. Otherwise we are muddying the waters\footnote{This is of course true for ker pc.}. The idea of Smith normal form is to compensate for the lost freedom by admitting also column versions of the operations 1.) and 2.). 
The operations 1.) and 2.) for rows can be understood by left multiplications with matrices in $\operatorname{GL}_{n}( R)$ while the operations 1.) and 2.) for columns correspond to right multiplications with matrices in $\operatorname{GL}_{m}( R)$. Obviously, the matrices used for the operations of type 1.) are transpositions $\boldsymbol{1} -E_{ii} -E_{jj} +E_{ij} +E_{ji}$ while the matrices used for the operations of type 2.) are of the form $\boldsymbol{1} +rE_{ij}$ with $r\in R$, the latter actually being unimodular.

\begin{theorem}[H.J.S. Smith (1861) \cite{Smith}]
\label{thm:Smith}
Let $n,m\geq 1$ be integers, $R$ be PID and $A\in R^{n\times m}$ a matrix of rank $p$. Then there exist  matrices $U\in \operatorname{GL}_{n}( R)$ and $V\in \operatorname{GL}_{m}( R)$ and a diagonal matrix $D=\sum _{i=1}^{p} d_{i} E_{ii} \in R^{n\times m}$ such that
\begin{enumerate}[label=\alph*.)]
    \item $d_{i} \neq 0$ for $i=1,\dotsc ,p$,
  \item  $d_{i} \mid d_{i+1}$ for $i=1,\dotsc ,p-1$,
  \item  $D=UAV$.
\end{enumerate}
The $d_{1} ,\dotsc ,d_{p}$ are unique up to multiplication by a unit and are called \emph{ invariant factors}.
\end{theorem}

For simplicity we will only discuss in detail the case when $R$ be a euclidean domain with euclidean function $\delta :R\mathbb{\backslash }\{0\}\rightarrow \{1,2,3,\dotsc \}$. We need also the following obvious consequence of the Smith normal form.

\begin{corollary} 
\label{cor:3submodules}
With the notations of Theorem \ref{thm:Smith} we have the following.
\begin{enumerate}
    \item A basis for the free $R$-module $\operatorname{im}_R( A)$ is provided by the first $p$ columns of $U^{-1} D$. 
    \item If $m >p$, the last $m-p$ columns of $V$ form a basis for the free $R$-module $\ker_R( A)$.
    \item The last $n-p$ rows of $U$ form a $(n-p)\times n$ matrix $B$ of rank $n-p$ such that $BA=\boldsymbol{0}$.
\end{enumerate}
\end{corollary}

For a detailed proof of Theoren \ref{thm:Smith}
 we refer to Jacobson \cite[Section 3.7]{JacobsonBA1}. For the reader to understand the basic ideas I work out some examples for the case $R=\mathbb{Z}$. Here the euclidean function is $\delta =|\ |$. I use again the notational trick of adding block matrices to keep track of the row and column operations. Other people have found this before me (see, e.g., the method of Gauß-Jordan for matrix inversion), but I do not know a reference where it is used for Smith normal form calculations. The trick has also the advantage that one can do check sums for rows and columns to spot errors.

\begin{example}
\label{ex:countfrom3}
We would like to find a Smith normal form of the matrix 
\begin{align*}
A=\begin{bmatrix}
3 & 4 & 5 & 6\\
7 & 8 & 9 & 10\\
11 & 12 & 13 & 14\\
15 & 16 & 17 & 18
\end{bmatrix}.\end{align*} The first step consists of switching a non-zero entry with minimal $\delta $ to position $(1,1)$ using operations of type 1.). Here it is not necessary since that entry \ $3$ is already there. Next, we want to annihilate everything left of $3$ using column operations of type 2.). As $3\nmid 4,5$ we cannot do it in one step but use division by $3$ with remainder:
\begin{align*}\begin{array}{|c c c c|c c c c|}
\hline
3 & 4 & 5 & 6 & 1 &  &  & \\
7 & 8 & 9 & 10 &  & 1 &  & \\
11 & 12 & 13 & 14 &  &  & 1 & \\
15 & 16 & 17 & 18 &  &  &  & 1\\
\hline
1 &  &  &  &  &  &  & \\
 & 1 &  &  &  &  &  & \\
 &  & 1 &  &  &  &  & \\
 &  &  & 1 &  &  &  & \\
\hline
\end{array} \mapsto \begin{array}{|c c c c|c c c c|}
\hline
3 & 1 & 2 &  & 1 &  &  & \\
7 & 1 & 2 & -4 &  & 1 &  & \\
11 & 1 & 2 & -8 &  &  & 1 & \\
15 & 1 & 2 & -12 &  &  &  & 1\\
\hline
1 & -1 & -1 & -2 &  &  &  & \\
 & 1 &  &  &  &  &  & \\
 &  & 1 &  &  &  &  & \\
 &  &  & 1 &  &  &  & \\
\hline
\end{array}.\end{align*}
The unit matrix in the bottom keeps track of the column operations and, as a matter of principle, we only work in the first block. Now we switch the first two columns so that an entry with minimal $\delta $ appears at position $( 1,1)$. This enable us to finish the clean up of the first row using column operations of type 2.).
\begin{align*}\mapsto \begin{array}{|c c c c|c c c c|}
\hline
1 & 3 & 2 &  & 1 &  &  & \\
1 & 7 & 2 & -4 &  & 1 &  & \\
1 & 11 & 2 & -8 &  &  & 1 & \\
1 & 15 & 2 & -12 &  &  &  & 1\\
\hline
-1 & 1 & -1 & -2 &  &  &  & \\
1 &  &  &  &  &  &  & \\
 &  & 1 &  &  &  &  & \\
 &  &  & 1 &  &  &  & \\
\hline
\end{array} \mapsto \begin{array}{|c c c c|c c c c|}
\hline
1 &  &  &  & 1 &  &  & \\
1 & 4 &  & -4 &  & 1 &  & \\
1 & 8 &  & -8 &  &  & 1 & \\
1 & 12 &  & -12 &  &  &  & 1\\
\hline
-1 & 4 & 1 & -2 &  &  &  & \\
1 & -3 & -2 &  &  &  &  & \\
 &  & 1 &  &  &  &  & \\
 &  &  & 1 &  &  &  & \\
\hline
\end{array}.\end{align*} 
Next we annihilate everything below the $1$ at position $( 1,1)$ using row operations of type 2.). The unit matrix on the right keeps track of the row operations.
\begin{align*}\mapsto \begin{array}{|c c c c|c c c c|}
\hline
1 &  &  &  & 1 &  &  & \\
 & 4 &  & -4 & -1 & 1 &  & \\
 & 8 &  & -8 & -1 &  & 1 & \\
 & 12 &  & -12 & -1 &  &  & 1\\
\hline
-1 & 4 & 1 & -2 &  &  &  & \\
1 & -3 & -2 &  &  &  &  & \\
 &  & 1 &  &  &  &  & \\
 &  &  & 1 &  &  &  & \\
\hline
\end{array}.\end{align*}
Since $1\mid 4$ we can proceed with the block $\begin{array}{ c c c }
4 & 0 & -4\\
8 & 0 & -8\\
12 & 0 & -12
\end{array}$($1$st row and column become spectators). An entry $4$ with minimal $\delta $ appear already at the right spot and we can clean what is on the right of it using a column operation of type 2.).
\begin{align*}\mapsto \begin{array}{|c c c c|c c c c|}
\hline
1 &  &  &  & 1 &  &  & \\
 & 4 &  &  & -1 & 1 &  & \\
 & 8 &  &  & -1 &  & 1 & \\
 & 12 &  &  & -1 &  &  & 1\\
\hline
-1 & 4 & 1 & 2 &  &  &  & \\
1 & -3 & -2 & -3 &  &  &  & \\
 &  & 1 &  &  &  &  & \\
 &  &  & 1 &  &  &  & \\
\hline
\end{array}.\end{align*}
Using row operations of type 2.) we arrive at 
\begin{align*}\mapsto \begin{array}{|c c c c|c c c c|}
\hline
1 &  &  &  & 1 &  &  & \\
 & 4 &  &  & -1 & 1 &  & \\
 &  &  &  & 1 & -2 & 1 & \\
 &  &  &  & 2 & -3 &  & 1\\
\hline
-1 & 4 & 1 & 2 &  &  &  & \\
1 & -3 & -2 & -3 &  &  &  & \\
 &  & 1 &  &  &  &  & \\
 &  &  & 1 &  &  &  & \\
\hline
\end{array} =\begin{array}{|c|c|}
\hline
D & U\\
\hline
V & \\
\hline
\end{array}.
\end{align*}
Hence 
\begin{align*}
A=\begin{bmatrix}
3 & 4 & 5 & 6\\
7 & 8 & 9 & 10\\
11 & 12 & 13 & 14\\
15 & 16 & 17 & 18
\end{bmatrix} =\begin{bmatrix}
1 & 0 & 0 & 0\\
-1 & 1 & 0 & 0\\
1 & -2 & 1 & 0\\
2 & -3 & 0 & 1
\end{bmatrix}^{-1}\begin{bmatrix}
1 & 0 & 0 & 0\\
0 & 4 & 0 & 0\\
0 & 0 & 0 & 0\\
0 & 0 & 0 & 0
\end{bmatrix}\begin{bmatrix}
-1 & 4 & 1 & 2\\
1 & -3 & -2 & -3\\
0 & 0 & 1 & 0\\
0 & 0 & 0 & 1
\end{bmatrix}^{-1} =U^{-1}DV^{-1}.
\end{align*}
We use inv pc to invert $U$:
\begin{align*}
\begin{array}{ c c c c|c c c c|c }
1 &  &  &  & 1 &  &  &  & \\
-1 & 1 &  &  &  & 1 &  &  & \\
1 & -2 & 1 &  &  &  & 1 &  & \\
2 & -3 &  & 1 &  &  &  & 1 & \\
\hline
1 &  &  &  &  &  &  &  & -1\\
 & 1 &  &  &  &  &  &  & -1\\
 &  & 1 &  &  &  &  &  & -1\\
 &  &  & 1 &  &  &  &  & -1
\end{array} \mapsto \begin{array}{ c c c|c c c c|c }
1 &  &  & 1 & 1 &  &  & \\
-2 & 1 &  & -1 &  & 1 &  & \\
-3 &  & 1 & -2 &  &  & 1 & \\
\hline
 &  &  & -1 &  &  &  & -1\\
1 &  &  &  &  &  &  & -1\\
 & 1 &  &  &  &  &  & -1\\
 &  & 1 &  &  &  &  & -1
\end{array} \mapsto \begin{array}{ c c|c c c c|c }
1 &  & 1 & 2 & 1 &  & \\
 & 1 & 1 & 3 &  & 1 & \\
\hline
 &  & -1 &  &  &  & -1\\
 &  & -1 & -1 &  &  & -1\\
1 &  &  &  &  &  & -1\\
 & 1 &  &  &  &  & -1
\end{array}\\
\mapsto \begin{array}{ c|c c c c|c }
1 & 1 & 3 &  & 1 & \\
\hline
 & -1 &  &  &  & -1\\
 & -1 & -1 &  &  & -1\\
 & -1 & -2 & -1 &  & -1\\
1 &  &  &  &  & -1
\end{array} \mapsto \begin{array}{|c c c c|c }
\hline
-1 &  &  &  & -1\\
-1 & -1 &  &  & -1\\
-1 & -2 & -1 &  & -1\\
-1 & -3 &  & -1 & -1
\end{array} \mapsto \begin{array}{|c c c c|c }
\hline
1 &  &  &  & 1\\
1 & 1 &  &  & 1\\
1 & 2 & 1 &  & 1\\
1 & 3 &  & 1 & 1
\end{array},\end{align*}
 i.e., $\begin{bmatrix}
1 & 0 & 0 & 0\\
-1 & 1 & 0 & 0\\
1 & -2 & 1 & 0\\
2 & -3 & 0 & 1
\end{bmatrix}^{-1} =\begin{bmatrix}
1 & 0 & 0 & 0\\
1 & 1 & 0 & 0\\
1 & 2 & 1 & 0\\
1 & 3 & 0 & 1
\end{bmatrix}$,
and conclude
\begin{align*}\ker_\Z\left(\begin{bmatrix}
3 & 4 & 5 & 6\\
7 & 8 & 9 & 10\\
11 & 12 & 13 & 14\\
15 & 16 & 17 & 18
\end{bmatrix}\right) =\mathbb{Z}\begin{bmatrix}
1\\
-2\\
1\\
0
\end{bmatrix} \oplus \mathbb{Z}\begin{bmatrix}
2\\
-3\\
0\\
1
\end{bmatrix} ,\ \operatorname{im}_\Z\left(\begin{bmatrix}
3 & 4 & 5 & 6\\
7 & 8 & 9 & 10\\
11 & 12 & 13 & 14\\
15 & 16 & 17 & 18
\end{bmatrix}\right) =\mathbb{Z}\begin{bmatrix}
1\\
1\\
1\\
1
\end{bmatrix} \oplus 
\mathbb{Z}\begin{bmatrix}
0\\
4\\
8\\
12
\end{bmatrix}.
\end{align*}
Transposing the Smith normal form $D^T=V^TA^TU^T$ one observes that in this example
$\ker_\Z(A)=\ker_\Z(A^T)$. To find a basis for the $A^T\Z^4$ one has to 
take the first 2 columns of $V^{-1}$. We leave this calculation to the interested reader and say for the record that the solution of the problem of the four submodules with Smith normal form requires two matrix inversions, one for the image of $A$ and one for the image of $A^T$. 
\end{example}

\begin{example}
There is a subtlety that we did not have to address in Example \ref{ex:countfrom3}. In order to realize condition b.) we occasionally have to work a bit more. Let us calculate the Smith normal form of the matrix
\begin{align*}
A=\begin{bmatrix}
0 & 0 & 10 & 0\\
-2 & 2 & -6 & 4\\
2 & 2 & 6 & 8
\end{bmatrix}.
\end{align*}
I am following here essentially \cite[Subsection 4.5]{Ash}. 
We switch the first and the third row to move a $2$ to position $( 1,1)$:
\begin{align*}\begin{array}{|c c c c|c c c|}
\hline
 &  & 10 &  & 1 &  & \\
-2 & 2 & -6 & -4 &  & 1 & \\
2 & 2 & 6 & 8 &  &  & 1\\
\hline
1 &  &  &  &  &  & \\
 & 1 &  &  &  &  & \\
 &  & 1 &  &  &  & \\
 &  &  & 1 &  &  & \\
\hline
\end{array} \mapsto \begin{array}{|c c c c|c c c|}
\hline
2 & 2 & 6 & 8 &  &  & 1\\
-2 & 2 & -6 & -4 &  & 1 & \\
 &  & 10 &  & 1 &  & \\
\hline
1 &  &  &  &  &  & \\
 & 1 &  &  &  &  & \\
 &  & 1 &  &  &  & \\
 &  &  & 1 &  &  & \\
\hline
\end{array}.
\end{align*}
Next we clean up everything to the right of that $2$, and afterwards, what is below of it:
\begin{align*}\mapsto \begin{array}{|c c c c|c c c|}
\hline
2 &  &  &  &  &  & 1\\
-2 & 4 &  & 4 &  & 1 & \\
 &  & 10 &  & 1 &  & \\
\hline
1 & -1 & -3 & -4 &  &  & \\
 & 1 &  &  &  &  & \\
 &  & 1 &  &  &  & \\
 &  &  & 1 &  &  & \\
\hline
\end{array} \mapsto \begin{array}{|c c c c|c c c|}
\hline
2 &  &  &  &  &  & 1\\
 & 4 &  & 4 &  & 1 & 1\\
 &  & 10 &  & 1 &  & \\
\hline
1 & -1 & -3 & -4 &  &  & \\
 & 1 &  &  &  &  & \\
 &  & 1 &  &  &  & \\
 &  &  & 1 &  &  & \\
\hline
\end{array}.\end{align*}
Since $2\mid 4$ we are fine and can move to the block $\begin{array}{ c c c }
4 & 0 & 4\\
0 & 10 & 0
\end{array}$, starting by cleaning up what is to the right of the $4$:
\begin{align*}\mapsto \begin{array}{|c c c c|c c c|}
\hline
2 &  &  &  &  &  & 1\\
 & 4 &  &  &  & 1 & 1\\
 &  & 10 &  & 1 &  & \\
\hline
1 & -1 & -3 & -3 &  &  & \\
 & 1 &  & -1 &  &  & \\
 &  & 1 &  &  &  & \\
 &  &  & 1 &  &  & \\
\hline
\end{array}
\end{align*}
Now the problem is that $4\nmid 10$ (this is the sublety). We add the third row to the second and do column operation of type 2.) dividing by $4$ with remainder:
\begin{align*}\mapsto \begin{array}{|c c c c|c c c|}
\hline
2 &  &  &  &  &  & 1\\
 & 4 & 10 &  & 1 & 1 & 1\\
 &  & 10 &  & 1 &  & \\
\hline
1 & -1 & -3 & -3 &  &  & \\
 & 1 &  & -1 &  &  & \\
 &  & 1 &  &  &  & \\
 &  &  & 1 &  &  & \\
\hline
\end{array} \mapsto \begin{array}{|c c c c|c c c|}
\hline
2 &  &  &  &  &  & 1\\
 & 4 & 2 &  & 1 & 1 & 1\\
 &  & 10 &  & 1 &  & \\
\hline
1 & -1 & -1 & -3 &  &  & \\
 & 1 & -2 & -1 &  &  & \\
 &  & 1 &  &  &  & \\
 &  &  & 1 &  &  & \\
\hline
\end{array}.
\end{align*}
Next we have to move the $2$ to position $(2,2)$ and clean up what is on the right of it:
\begin{align*}\mapsto \begin{array}{|c c c c|c c c|}
\hline
2 &  &  &  &  &  & 1\\
 & 2 & 4 &  & 1 & 1 & 1\\
 & 10 &  &  & 1 &  & \\
\hline
1 & -1 & -1 & -3 &  &  & \\
 & -2 & 1 & -1 &  &  & \\
 & 1 &  &  &  &  & \\
 &  &  & 1 &  &  & \\
\hline
\end{array} \mapsto \begin{array}{|c c c c|c c c|}
\hline
2 &  &  &  &  &  & 1\\
 & 2 &  &  & 1 & 1 & 1\\
 & 10 & -20 &  & 1 &  & \\
\hline
1 & -1 & 1 & -3 &  &  & \\
 & -2 & 5 & -1 &  &  & \\
 & 1 & -2 &  &  &  & \\
 &  &  & 1 &  &  & \\
\hline
\end{array}.
\end{align*}
It remains to clean up what is below that $2$. For aesthetic reasons we multiply the 3rd row by $-1$:
\begin{align*}\mapsto \begin{array}{|c c c c|c c c|}
\hline
2 &  &  &  &  &  & 1\\
 & 2 &  &  & 1 & 1 & 1\\
 &  & -20 &  & -4 & -5 & -5\\
\hline
1 & -1 & 1 & -3 &  &  & \\
 & -2 & 5 & -1 &  &  & \\
 & 1 & -2 &  &  &  & \\
 &  &  & 1 &  &  & \\
\hline
\end{array} \mapsto \begin{array}{|c c c c|c c c|}
\hline
2 &  &  &  &  &  & 1\\
 & 2 &  &  & 1 & 1 & 1\\
 &  & 20 &  & 4 & 5 & 5\\
\hline
1 & -1 & 1 & -3 &  &  & \\
 & -2 & 5 & -1 &  &  & \\
 & 1 & -2 &  &  &  & \\
 &  &  & 1 &  &  & \\
\hline
\end{array} =\begin{array}{|c|c|}
\hline
D & U\\
\hline
V & \\
\hline
\end{array}.
\end{align*}
Hence 
\begin{align*}A=\begin{bmatrix}
0 & 0 & 10 & 0\\
-2 & 2 & -6 & -4\\
2 & 2 & 6 & 8
\end{bmatrix} =\begin{bmatrix}
0 & 0 & 1\\
1 & 1 & 1\\
4 & 5 & 5
\end{bmatrix}^{-1}\begin{bmatrix}
2 & 0 & 0 & 0\\
0 & 2 & 0 & 0\\
0 & 0 & 20 & 0
\end{bmatrix}\begin{bmatrix}
1 & -1 & 1 & -3\\
0 & -2 & 5 & -1\\
0 & 1 & -2 & 0\\
0 & 0 & 0 & 1
\end{bmatrix}^{-1}=U^{-1}DV^{-1}.
\end{align*}
We use inv pc to invert $U$:
\begin{align*}\begin{array}{ c c c|c c c|c }
 &  & 1 & 1 &  &  & \\
1 & 1 & 1 &  & 1 &  & \\
4 & 5 & 5 &  &  & 1 & \\
\hline
1 &  &  &  &  &  & -1\\
 & 1 &  &  &  &  & -1\\
 &  & 1 &  &  &  & -1
\end{array} \mapsto \begin{array}{ c c c|c c c|c }
1 & 1 & 1 &  & 1 &  & \\
4 & 5 & 5 &  &  & 1 & \\
 &  & 1 & 1 &  &  & \\
\hline
1 &  &  &  &  &  & -1\\
 & 1 &  &  &  &  & -1\\
 &  & 1 &  &  &  & -1
\end{array} \mapsto \begin{array}{ c c|c c c|c }
1 & 1 &  & -4 & 1 & \\
 & 1 & 1 &  &  & \\
\hline
-1 & -1 &  & -1 &  & -1\\
1 &  &  &  &  & -1\\
 & 1 &  &  &  & -1
\end{array}\\
\mapsto \begin{array}{ c|c c c|c }
1 & 1 &  &  & \\
\hline
 &  & -5 & 1 & -1\\
-1 &  & 4 & -1 & -1\\
1 &  &  &  & -1
\end{array} \mapsto \begin{array}{|c c c|c }
\hline
 & -5 & -1 & -1\\
1 & 4 & -1 & -1\\
-1 &  &  & -1
\end{array} \mapsto \begin{array}{|c c c|c }
\hline
 & 5 & -1 & 1\\
-1 & -4 & 1 & 1\\
1 &  &  & 1
\end{array},\end{align*} 
i.e.,
$\begin{bmatrix}
0 & 0 & 1\\
1 & 1 & 1\\
4 & 5 & 5
\end{bmatrix}^{-1} =\begin{bmatrix}
0 & 5 & -1\\
-1 & -4 & 1\\
1 & 0 & 0
\end{bmatrix}$, and hence deduce
\begin{align*}\ker_\Z\left(\begin{bmatrix}
0 & 0 & 10 & 0\\
-2 & 2 & -6 & -4\\
2 & 2 & 6 & 8
\end{bmatrix}\right) =\mathbb{Z}\begin{bmatrix}
-3\\
-1\\
0\\
1
\end{bmatrix} ,\ \operatorname{im}_\Z\left(\begin{bmatrix}
0 & 0 & 10 & 0\\
-2 & 2 & -6 & -4\\
2 & 2 & 6 & 8
\end{bmatrix}\right) =\mathbb{Z}\begin{bmatrix}
0\\
-2\\
2
\end{bmatrix} \oplus \mathbb{Z}\begin{bmatrix}
10\\
-8\\
0
\end{bmatrix} \oplus \mathbb{Z}\begin{bmatrix}
-20\\
20\\
0
\end{bmatrix}.
\end{align*}
\end{example}

\begin{example}
Let us quickly re-examine Example \ref{ex:rost}. We would like to understand the kernel of 
\begin{align*}
A=\begin{bmatrix}
1 & 0 & 1 & 2\\
0 & 2 & 1 & 3
\end{bmatrix}
\end{align*} 
as a $\mathbb{Z}$-module using a Smith normal form:
\begin{align*}\begin{array}{|c c c c|c c|}
\hline
1 &  & 1 & 2 & 1 & \\
 & 2 & 1 & 3 &  & 1\\
\hline
1 &  &  &  &  & \\
 & 1 &  &  &  & \\
 &  & 1 &  &  & \\
 &  &  & 1 &  & \\
\hline
\end{array} \mapsto \begin{array}{|c c c c|c c|}
\hline
1 &  &  &  & 1 & \\
 & 2 & 1 & 3 &  & 1\\
\hline
1 &  & -1 & -2 &  & \\
 & 1 &  &  &  & \\
 &  & 1 &  &  & \\
 &  &  & 1 &  & \\
\hline
\end{array} \mapsto \begin{array}{|c c c c|c c|}
\hline
1 &  &  &  & 1 & \\
 & 1 & 2 & 3 &  & 1\\
\hline
1 & -1 &  & -2 &  & \\
 &  & 1 &  &  & \\
 & 1 &  &  &  & \\
 &  &  & 1 &  & \\
\hline
\end{array} \mapsto \begin{array}{|c c c c|c c|}
\hline
1 &  &  &  & 1 & \\
 & 1 &  &  &  & 1\\
\hline
1 & -1 & 2 & 1 &  & \\
 &  & 1 &  &  & \\
 & 1 & -2 & -3 &  & \\
 &  &  & 1 &  & \\
\hline
\end{array}.
\end{align*}
Hence
\begin{align*}\begin{bmatrix}
1 & 0 & 1 & 2\\
0 & 2 & 1 & 3
\end{bmatrix} =\begin{bmatrix}
1 & 0\\
0 & 1
\end{bmatrix}^{-1}\begin{bmatrix}
1 & 0 & 0 & 0\\
0 & 1 & 0 & 0
\end{bmatrix}\begin{bmatrix}
1 & -1 & 2 & 1\\
0 & 0 & 1 & 0\\
0 & 1 & -2 & -3\\
0 & 0 & 0 & 1
\end{bmatrix}^{-1} ,\ \ker_\Z\left(\begin{bmatrix}
1 & 0 & 1 & 2\\
0 & 2 & 1 & 3
\end{bmatrix}\right) =\mathbb{Z}\begin{bmatrix}
2\\
1\\
-2\\
0
\end{bmatrix} \oplus \mathbb{Z}\begin{bmatrix}
1\\
0\\
-3\\
1
\end{bmatrix}.
\end{align*}

\end{example}

We leave it to the reader to elaborate the version of what has been said above for the case when $R$ is a PID (see \cite[Section 3.7]{JacobsonBA1}). Note that one has to add an extra operation from Bezout's theorem. We need one more detail about the Smith normal form (see \cite[Theorem 3.9]{JacobsonBA1}).
\begin{theorem}
\label{thm:minors}
With the assumption of Theorem \ref{thm:Smith} let $\Delta _{i}$ be the gcd of the $i\times i$-minors of $A$. Then $d_{1} =\Delta _{1}$ and $d_{i} =\frac{\Delta _{i}}{\Delta _{i-1}}$ for $i=2,\dotsc ,p$.
\end{theorem}

\subsection{Saturated submodules of $ R^{n}$}

I am pretty sure the following material is written up  somewhere. I just could not allocate it.

Let $ S$ be a multiplicative subset of the domain $R$ and $S^{-1}R$ the localization of $R$ at $S$. An $R$-module $ W$ of $ R^{n}$ is called $ S$\emph{-saturated} if for any $ r\in S$ we have that $ r\vec{w} \in W$ implies $ \vec{w} \in W$. Let us denote by $ f$ the canonical $R$-linear map $ R^{n}\rightarrow \left( S^{-1} R\right)^{n}$. It is easy to verify that an $ R$-module $ W$ is $S$-saturated if and only if there is an $S^{-1} R$-submodule $ M$ of $ \left( S^{-1} R\right)^{n}$ \ such that $ W=f^{-1}( M)$. More generally, the $ S$\emph{-saturation} of $ W$
is $ f^{-1}( f( W))$.

We now specialize to the case $ S=R\backslash \{0\}$. Then $ S^{-1} R=Q$ is the field of fractions of $ R$. An $(R\backslash \{0\})$-saturated submodule $ W$ of $ R^{n}$ we simply call \emph{saturated}. The key observation is that, since $ R$ is a domain, the kernel of a matrix $ A\in R^{n\times m}$ is saturated. This is because $A( r\vec{w}) =rA(\vec{w}) =0$ implies $ A(\vec{w}) =0$ for any $ r\in R\backslash \{0\}$.
Since an inclusion of two equal dimensional $Q$-vector spaces is an equality we have the following.
\begin{lemma}
Let $R$ be a domain and $W$ an $R$-submodule of the saturated $R$-submodule $M\subseteq R^n$ such that $ \operatorname{rank} W=\operatorname{rank} M$. Then $ W=M$ if and only $ W$ is saturated.
\end{lemma}


\begin{proposition}
\label{prop:wedge}
Let $R$ be a domain and $W$ be a $R$-submodule of $R^{n}$ and $l\geq 0$ and integer.
If $W$ saturated then $\land_R^{l} W$ saturated for each $l\geq 0$.
\end{proposition}

\begin{proof}
We have $W=f^{-1}( V)$ \ for a $Q$-vector subspace $V\subseteq Q^{n}$, i.e., 
for every $\vec{w} \in W$ exists $v\in V$: $f(\vec{w}) =\vec{v}$.
Every $\omega \in \land ^{l} W$ can be written as $\omega =\vec{w}_{1} \land \dotsc \land \vec{w}_{l}$ for some $\vec{w}_{1} ,\dotsc ,\vec{w}_{l} \in W$ and there exist $\vec{v}_{1} ,\dotsc ,\vec{v}_{l} \in V$ such that $f(\vec{w}_{i}) =\vec{v}_{i}$. Hence $\left( \land ^{l}_R f\right)(\vec{w}_{1} \land \dotsc \land \vec{w}_{l}) =\vec{v}_{1} \land \dotsc \land \vec{v}_{l}$ and $\land ^{l}_R W=\left( \land ^{l}_R f\right)^{-1}\left( \land ^{l}_Q V\right)$.
\end{proof}

If $R$ is a PID we can be more specific.

\begin{theorem}
\label{thm:characterization}
Let $R$ be a PID and $X\in R^{n\times p}$ a matrix of rank $p$ with invariant factors $d_{1} ,\dotsc ,d_{p}$. Then the following statements are equivalent:
\begin{enumerate}[label=\roman*.)]
    \item $\operatorname{im}_R(X)$ is saturated.
\item $d_{i} \in R^{\times }$ for all $i=1,\dotsc ,p$.
\item $\Delta _{p} \in R^{\times }$.
\end{enumerate}
\end{theorem}

\begin{proof}
The implications i.)$\Longleftrightarrow $ii.) $\Longrightarrow $iii.) are obvious from Theorems \ref{thm:Smith} and \ref{thm:minors}. To establish the implication iii.)$ \Longrightarrow $ ii.) we show that if $d_{p} \notin R^{\times }$ it follows that $\Delta _{p} \notin R^{\times }$.
If $A\in R^{r\times s}$ is a matrix we denote by $A_{i_{1} \dotsc i_{p}}^{j_{1} \dotsc j_{p}}$ the $p\times p$-minor corresponding the row indices $1\leq i_{1} < \dotsc < i_{p} \leq r$ and column indices $1\leq j_{1} < \dotsc < j_{p} \leq s$. From the Smith normal form $X=U^{-1} DV^{-1}$ and the Cauchy-Binet formula we deduce
\begin{align*}
X_{i_{1} \dotsc i_{p}}^{1\dotsc p} =\sum _{1\leq i_{1} < \dotsc < i_{p} \leq n}\left( U^{-1}\right)_{i_{1} \dotsc i_{p}}^{1\dotsc p} \ \ D_{1\dotsc p}^{1\dotsc p} \ \left( V^{-1}\right)_{1\dotsc p}^{1\dotsc p} =\left(\prod _{i=1}^{p} d_{i}\right)\sum _{1\leq i_{1} < \dotsc < i_{p} \leq n}\left( U^{-1}\right)_{i_{1} \dotsc i_{p}}^{1\dotsc p} \ \ \left( V^{-1}\right)_{1\dotsc p}^{1\dotsc p}.
\end{align*}
But this means that $d_{p} \mid X_{i_{1} \dotsc i_{p}}^{1\dotsc p}$ and the gcd $\Delta _{p}$ of the $X_{i_{1} \dotsc i_{p}}^{1\dotsc p}$ cannot be a unit.
\end{proof}

\begin{remark}
    \label{rem:1dim}
    Let $R$ be a gcd domain (e.g., a PID).
Let $V$ be a one-dimensional $Q$-vector space. Clearing the denominators we write $V=Q\vec w$ for some $\vec w\in R^n$. Put $W:=R\vec w$ and note that $V=Q\otimes_R W$. Now $W$ is saturated if and only if the gcd $\Delta_1$ of the entries of $\vec w$ is a unit. Such vectors are also called \emph{primitive} in the literature \cite{BAY}. The generators that are calculated by ker pc are by construction primitive. That is, if $\dim(\ker_Q(A))=1$ then ker pc provides a generator for $\ker_R(A)$.
\end{remark}

\subsection{Turning a $Q$-basis into an $R$-basis}
\label{subsec:Rbasis}

In this section we assume $R$ is a PID.
Once we have already calculated $Q$-basis $\vec{v}_{1} ,\dotsc ,\vec{v}_{p}\in R^n$ for the $Q$-vector space $\ker_Q( A)$ for $A\in R^{n\times m}$ using ker pc it seems desirable to turn it into an $R$-basis. In particular, if $n,m$ are large but $\operatorname{rank}(\ker( A))$ is small it is not very attractive to go through the whole Smith normal form calculation for the matrix $A$.


Here one can proceed as follows. Take 
$Y_{p}^T=\begin{bmatrix}
\vec{v}_1 &\dots &\vec{v}_p
\end{bmatrix}\in R^{n\times p}$
 and calculate its Smith normal form $D'=U'Y_{p}^TV'$. Form the matrix $B$ by taking the last $n-p$ rows of $U'$ (see Corollary \ref{cor:3submodules}(3)) and calculate the Smith normal form $D''=U''BV''$ of $B$. The last $p$ columns of $V''$ form a basis for the $R$-module $\ker_R(A)$. More generally, this recipe can be used to calculate the $R$-saturation of a submodule of $R^n$.

\begin{example}
    Let us take a final look at Example \ref{ex:rost}. The kernel of 
\begin{align*}
A=\begin{bmatrix}
1 & 0 & 1 & 2\\
0 & 2 & 1 & 3
\end{bmatrix}\end{align*}  was shown by ker pc to be $ \begin{bmatrix}
-2 & -4\\
-1 & -3\\
2 & 0\\
0 & 2
\end{bmatrix}\mathbb{Q}^{2}$ as a $ \mathbb{Q}$-vector space. We now use a Smith normal form to calculate $ B$: 
\begin{align*}
\begin{array}{|c c|c c c c|}
\hline
-2 & -4 & 1 &  &  & \\
-1 & -3 &  & 1 &  & \\
2 &  &  &  & 1 & \\
 & 2 &  &  &  & 1\\
\hline
1 &  &  &  &  & \\
 & 1 &  &  &  & \\
\hline
\end{array} \mapsto \begin{array}{|c c|c c c c|}
\hline
-1 & -3 &  & 1 &  & \\
-2 & -4 & 1 &  &  & \\
2 &  &  &  & 1 & \\
 & 2 &  &  &  & 1\\
\hline
1 &  &  &  &  & \\
 & 1 &  &  &  & \\
\hline
\end{array} \mapsto \begin{array}{|c c|c c c c|}
\hline
-1 & -3 &  & 1 &  & \\
-2 & -4 & 1 &  &  & \\
2 &  &  &  & 1 & \\
  & 2 &  &  &  & 1\\
\hline
1 &  &  &  &  & \\
 & 1 &  &  &  & \\
\hline
\end{array}\\
\mapsto \begin{array}{|c c|c c c c|}
\hline
-1 &  &  & 1 &  & \\
-2 & 2 & 1 &  &  & \\
2 & -6 &  &  & 1 & \\
 & 2 &  &  &  & 1\\
\hline
1 & -3 &  &  &  & \\
 & 1 &  &  &  & \\
\hline
\end{array} \mapsto \begin{array}{|c c|c c c c|}
\hline
-1 &  &  & 1 &  & \\
 & 2 & 1 & 2 &  & \\
 & -6 &  & -2 & 1 & \\
 & 2 &  &  &  & 1\\
\hline
1 & -3 &  &  &  & \\
 & 1 &  &  &  & \\
\hline
\end{array} \mapsto \begin{array}{|c c|c c c c|}
\hline
-1 &  &  & 1 &  & \\
 & 2 & 1 & 2 &  & \\
 &  & 3 & -4 & 1 & \\
 &  & -1 & 2 &  & 1\\
\hline
1 & -3 &  &  &  & \\
 & 1 &  &  &  & \\
\hline
\end{array},
\end{align*}
so that $ B=\begin{bmatrix}
3 & -4 & 1 & 0\\
-1 & 2 & 0 & 1
\end{bmatrix}$. Then we calculate a Smith normal form of $ B$
\begin{align*}
&\begin{array}{|c c c c|c c|}
\hline
3 & -4 & 1 &  & 1 & \\
-1 & 2 &  & 1 &  & 1\\
\hline
1 &  &  &  &  & \\
 & 1 &  &  &  & \\
 &  & 1 &  &  & \\
 &  &  & 1 &  & \\
\hline
\end{array} \mapsto \begin{array}{|c c c c|c c|}
\hline
-1 & 2 &  & 1 &  & 1\\
3 & -4 & 1 &  & 1 & \\
\hline
1 &  &  &  &  & \\
 & 1 &  &  &  & \\
 &  & 1 &  &  & \\
 &  &  & 1 &  & \\
\hline
\end{array} \mapsto \begin{array}{|c c c c|c c|}
\hline
-1 &  &  &  &  & 1\\
3 & 2 & 1 & 3 & 1 & \\
\hline
1 & 2 &  & 1 &  & \\
 & 1 &  &  &  & \\
 &  & 1 &  &  & \\
 &  &  & 1 &  & \\
\hline
\end{array}\\
&\mapsto \begin{array}{|c c c c|c c|}
\hline
-1 &  &  &  &  & 1\\
 & 2 & 1 & 3 & 1 & 3\\
\hline
1 & 2 &  & 1 &  & \\
 & 1 &  &  &  & \\
 &  & 1 &  &  & \\
 &  &  & 1 &  & \\
\hline
\end{array} \mapsto \begin{array}{|c c c c|c c|}
\hline
-1 &  &  &  &  & 1\\
 & 1 & 2 & 3 & 1 & 3\\
\hline
1 &  & 2 & 1 &  & \\
 &  & 1 &  &  & \\
 & 1 &  &  &  & \\
 &  &  & 1 &  & \\
\hline
\end{array} \mapsto \begin{array}{|c c c c|c c|}
\hline
-1 &  &  &  &  & 1\\
 & 1 &  &  & 1 & 3\\
\hline
1 &  & 2 & 1 &  & \\
 &  & 1 &  &  & \\
 & 1 & -2 & -3 &  & \\
 &  &  & 1 &  & \\
\hline
\end{array}
\end{align*}
and deduce once again $ \ker_\Z\left(\begin{bmatrix}
1 & 0 & 1 & 2\\
0 & 2 & 1 & 3
\end{bmatrix}\right) =\mathbb{Z}\begin{bmatrix}
2\\
1\\
-2\\
0
\end{bmatrix} \oplus \mathbb{Z}\begin{bmatrix}
1\\
0\\
-3\\
1
\end{bmatrix}$. 
\end{example}
In the example it looks like that the last method\footnote{How should one call it? I am voting for \emph{smitherate}.} to calculate the basis for $\ker_\Z(A)$ is more labour intensive than calculating a Smith normal form of $A$. But let us say, to be concrete, $A$ is a $100\times 100$ matrix and $\ker_\Z(A)$ is of rank $3$. Then ker pc seems faster than Smith normal form as it does not need division with remainder (it works with bigger numbers though). The saturation procedure explained here takes only about $6\%$ of the total effort and if, by accident, $\Delta_3$ of
$Y_3^T$  is in $R^\times$ it is not needed at all.

I do not know a method to obtain a basis for the non-saturated $\operatorname{im}_R(A)$
by post-processing the result of ker pc.
\bibliographystyle{amsalpha}
\bibliography{reactions}

\providecommand{\bysame}{\leavevmode\hbox to3em{\hrulefill}\thinspace}
\providecommand{\MR}{\relax\ifhmode\unskip\space\fi MR }
\providecommand{\MRhref}[2]{%
  \href{http://www.ams.org/mathscinet-getitem?mr=#1}{#2}
}
\providecommand{\href}[2]{#2}
\begin{thebibliography}{MKAM15}

\bibitem[Abe14]{Abeles}
Francine~F. Abeles, \emph{Chi\`o's and {D}odgson's determinantal identities},
  Linear Algebra Appl. \textbf{454} (2014), 130--137.

\bibitem[Ash07]{Ash}
Robert~B. Ash, \emph{Basic abstract algebra. {For} graduate students and
  advanced undergraduates}, Mineola, NY: Dover Publications, 2007 (English).

\bibitem[BAY21]{BAY}
Faten Ben~Amor and Ihsen Yengui, \emph{Saturation of finitely-generated
  submodules of free modules over {Pr{\"u}fer} domains}, Armen. J. Math.
  \textbf{13} (2021), 21, Id/No 1.

\bibitem[Chi53]{Chio}
F.~Chiò, \emph{Mémoire sur les fonctions connues sous le nom de résultantes
  ou de déterminants}, Turin, 1853.

\bibitem[Eve12]{Eves}
H.~Eves, \emph{Elementary matrix theory}, Dover Books on Mathematics, Dover
  Publications, 2012.

\bibitem[Fie11]{Fiedler}
Miroslav Fiedler, \emph{Matrices and graphs in geometry}, Encyclopedia of
  Mathematics and its Applications, vol. 139, Cambridge University Press,
  Cambridge, 2011.

\bibitem[Fin09]{Fink}
J.K. Fink, \emph{Physical chemistry in depth}, Springer Berlin Heidelberg,
  2009.

\bibitem[GP08]{Greuel}
Gert-Martin Greuel and Gerhard Pfister, \emph{A {\bf {s}ingular} introduction
  to commutative algebra}, extended ed., Springer, Berlin, 2008.

\bibitem[Gro]{Grossman}
\url{https://math.stackexchange.com/questions/4239718/kernel-of-a-matrix-using-the-rref-of-its-transpose}.

\bibitem[Jac85]{JacobsonBA1}
Nathan Jacobson, \emph{Basic algebra {I}. 2nd ed}, New {York}: {W}. {H}.
  {Freeman} and {Company}. {XVIII}, 499 p., 1985.

\bibitem[Jac89]{JacobsonBA2}
\bysame, \emph{Basic algebra {II}.}, 2nd ed. ed., New York, NY: W. H. Freeman
  {and} Company, 1989 (English).

\bibitem[MKAM15]{whynots}
Anna~Maria Michałowska-Kaczmarczyk, Agustin~G. Asuero, and Tadeusz
  Michałowski, \emph{“{W}hy {N}ot {S}toichiometry” versus
  “{S}toichiometry—{W}hy {N}ot?” part i: General context}, Critical
  Reviews in Analytical Chemistry \textbf{45} (2015), no.~2, 166--188, PMID:
  25558777.

\bibitem[Moo97]{Forester}
John~W. Moore, \emph{Balancing the forest and the trees}, J. Chem. Educ.
  \textbf{74} (1997), no.~11, 1253.

\bibitem[New72]{IntegralMatrices}
Morris Newman, \emph{Integral matrices}, Pure and {Applied} {Mathematics}, 45.
  {New} {York}-{London}: {Academic} {Press}. {XVII}, 224 p. \$ 14.00 (1972).,
  1972.

\bibitem[{OEI}]{OEIS}
{OEIS Foundation Inc.}, \emph{The {O}n-{L}ine {E}ncyclopedia of {I}nteger
  {S}equences}, Published electronically at \url{http://oeis.org}.

\bibitem[PV06]{Papp}
D{\'a}vid Papp and B{\'e}la Vizv{\'a}ri, \emph{Effective solution of linear
  {Diophantine} equation systems with an application in chemistry}, J. Math.
  Chem. \textbf{39} (2006), no.~1, 15--31 (English).

\bibitem[Ric68]{Richter}
Jeremias~Benjamin Richter, \emph{{A}nfangsgr{\"u}nde der {S}t{\"o}chiometrie:
  oder, {M}esskunst chemischer {E}lemente}, no. Bd. 1-2, G. Olms, 1968,
  Nachdruck der Ausgabe Breslau 1792-94.

\bibitem[Ris09]{IceB}
Ice~B. Risteski, \emph{A new singular matrix method for balancing chemical
  equations and their stability}, Journal of the Chinese Chemical Society
  \textbf{56} (2009), no.~1, 65--79.

\bibitem[Smi94]{Smith}
Henry John~Stephen Smith, \emph{The collected mathematical papers. {Edited} by
  {J}. {W}. {L}. {Glaisher}. {With} a mathematical introduction by the editor,
  biographical sketches and a portrait. {In} two volumes.}, Oxford: at the
  {Clarendon} {Press}. {Volume} {I}: xcv, 603 p. (with portrait); {Volume}
  {II}: vii, 719 p. (1894)., 1894.

\bibitem[Sto95]{Stout}
Roland Stout, \emph{Redox challenges: Good times for puzzle fanatics}, Journal
  of Chemical Education \textbf{72} (1995), no.~12, 1125.

\bibitem[Yen15]{Yengui}
Ihsen Yengui, \emph{Constructive commutative algebra. {Projective} modules over
  polynomial rings and dynamical {Gr{\"o}bner} bases}, Lect. Notes Math., vol.
  2138, Cham: Springer, 2015 (English).

\end{thebibliography}
\end{document}